\documentclass[12pt,thmsa]{amsart}
\usepackage{amsfonts}
\usepackage{amssymb, euscript, mathrsfs}

\newtheorem{theorem}{Theorem}[section]
\newtheorem{lemma}[theorem]{Lemma}

\newtheorem{definition}[theorem]{Definition}

\newtheorem{example}[theorem]{Example}

\newtheorem{conjecture}[theorem]{Conjecture}
\newtheorem{remark}[theorem]{Remark}

\newtheorem{corollary}[theorem]{Corollary}

\numberwithin{equation}{section}

\def\Ms{\mathscr{M m 1234}}

\def\Cs{\mathscr{C c 1234}}

\def\Ss{\mathscr{S s 1234}}

\def\Qs{\mathscr{Q q 1234}}

\def\Cal{\mathcal}

\def\R{{\Cal R}}
\def\C{{\Cal C}}

\def\P{{\Cal P}}

\def\S{{\Cal S}}
\def\F{{\Cal F}}

\def\I{{\Cal I}}

\def\d{\partial}
\def\tr{{\hbox{\rm tr}}}

\def\Ma{\frM_{n,m}}

\def\Mlm{\frM_{\ell,m}}

\def\gnk{G_{n,k}}

\def\f0{f_0}
\def\Fc0{\varphi_0}

\def\I_k {I_{-}^{k/2}}
\def\I+k {I_{+}^{k/2}}

\def\vnk{V_{n,k}}
\def\vnm{V_{n,m}}
\def\Gr{\frT}

\def\cd{\stackrel{*}{\C}\!{}_{m, k}^\a}
\def\sd{\stackrel{*}{\S}\!{}_{m, k}^\a}
\def\cd0{\stackrel{*}{\C}\!{}_{m, k}^\a}
\def\sd0{\stackrel{*}{\S}\!{}_{m, k}^\a}
\def\fd{\stackrel{*}{F}\!{}_{\!m, k}}
\def\rdk{\stackrel{*}{R}\!{}_{m, n-k}}

\def\ncd0{\stackrel{*}{\Cs}\!{}_{m, k}^\a}

\def\ncs{\stackrel{*}{\Ss}\!{}_{m, k}^\a}
\def\ncb{\stackrel{*}{\Ss}\!{}_{m, k}^\b}
\def\nck{\stackrel{*}{\S}\!{}_{m, k}^{\b+n-k-m}}

\def\bbr{{\Bbb R}}

\def\bbn{{\Bbb N}}

\def\bbc{{\Bbb C}}

\def\rank{{\hbox{\rm rank}}}

\def\tr{{\hbox{\rm tr}}}

\def\cos{{\hbox{\rm cos}}}
\def\det{{\hbox{\rm det}}}

\def\min{{\hbox{\rm min}}}

\def\Pr{{\hbox{\rm Pr}}}

\def\gnk{G_{n,k}}
\def\gnm{G_{n,m}}

\def\part{\partial}
\def\intl{\int\limits}
\def\b{\beta}

\def\Gam{\Gamma}
\def\Om{\Omega}
\def\a{\alpha}
\def\om{\omega}

\def\Del{\Delta}
\def\del{\delta}
\def\vp{\varphi}

\def\gam{\gamma}
\def\Lam{\Lambda}
\def\sig{\sigma}
\def\lam{\lambda}
\def\z{\zeta}

\def\t{\tau}

\def \nv{{\bf n}_0}

\def \kv{{\bf k}_0}
\def\lv{{\boldsymbol \lam}}
\def\mv{{\boldsymbol \mu}}
\def\av{{\boldsymbol \a}}

\font\frak=eufm10

\def\fr#1{\hbox{\frak #1}}

\def\frC{\fr{C}}        
\def\frD{\fr{D}}

\def\frL{\fr{L}}        
\def\frM{\fr{M}}

\def\frT{\fr{T}}

\def\pl{\P_\ell}

\def\vmk{V_{m,k}}

\def\cos{{\hbox{\rm cos}}}
\def\det{{\hbox{\rm det}}}

\def\min{{\hbox{\rm min}}}

\def\p{\P_m}
\def\gm{\Gamma_m}
\def\tr{{\hbox{\rm tr}}}
\def\part{\partial}
\def\intl{\int\limits}
\def\b{\beta}

\def\Gam{\Gamma}
\def\Om{\Omega}
\def\a{\alpha}

\def\sideremark#1{\ifvmode\leavevmode\fi\vadjust{\vbox to0pt{\vss
 \hbox to 0pt{\hskip\hsize\hskip1em
\vbox{\hsize2cm\tiny\raggedright\pretolerance10000
 \noindent #1\hfill}\hss}\vbox to8pt{\vfil}\vss}}}%

\newcommand{\edz}[1]{\sideremark{#1}}

\newcommand{\be}{\begin{equation}}
\newcommand{\ee}{\end{equation}}
\newcommand{\bea}{\begin{eqnarray}}
\newcommand{\eea}{\end{eqnarray}}
\newcommand{\Bea}{\begin{eqnarray*}}
\newcommand{\Eea}{\end{eqnarray*}}
\begin{document}

\title[  ]
{Funk, Cosine, and Sine Transforms on Stiefel and Grassmann manifolds}

\author{  B. Rubin }
\address{Department of Mathematics, Louisiana State University, Baton Rouge,
Louisiana 70803, USA}
\email{borisr@math.lsu.edu}

\thanks{ The work was
 supported  by the NSF grants PFUND-137 (Louisiana Board of Regents),
 DMS-0556157, and the Hebrew University of Jerusalem.}




\begin{abstract}
The Funk, cosine, and sine transforms on the unit sphere are indispensable tools in integral geometry. They are also known to be interesting objects in harmonic analysis.
The aim of the  paper is to extend basic facts about these transforms   to the  more general context for Stiefel or Grassmann manifolds. The main topics are composition formulas, the Fourier functional relations for the corresponding homogeneous distributions, analytic continuation, and explicit inversion formulas.

 \end{abstract}

\maketitle

\centerline {CONTENTS}

1. Introduction.

2. Preliminaries.

3. The higher-rank Funk transform.

4. Cosine  and sine transforms. Composition formulas.

5. Cosine  transforms via the Fourier analysis.

6. Normalized cosine and sine transforms.

7. The method of Riesz potentials.

8. Appendix.

\section{Introduction}
\setcounter{equation}{0}

 \subsection{History and motivation} Our consideration has several sources.

 {\bf 1.} For a function $\Phi$ on the unit sphere $S^2$ in $\bbr^3$ P. Funk \cite{F11, F13} defined a {\it circle-integral function } ({\it die Kreisintegral-Funktion}) $\chi$  on the set of great circles as an integral of $\Phi$ over the corresponding great circle.  He suggested two inversion algorithms; see
 \cite[pp. 285-288]{F13}. The first one relies on expansion in  spherical harmonics  and the second  reduces the problem to Abel's integral equation. From these results Funk
  derived the celebrated Minkowski's theorem \cite{Min}, \cite [p. 137]{Hel}, stating  that bodies of constant circumference are bodies of constant width.

  {\bf 2.}    Funk's Kreisintegral-Funktion  is now called {\it the Funk transform}. This concept extends to higher dimensions, when great circles on  $S^2$ are substituted by cross-sections $S^{n-1}\cap \xi$,  $\xi$ being a   $k$-dimensional linear subspace of $\bbr^n$, $1\le k\le n-1$. The set of all such subspaces forms a Grassmann manifold $G_{n,k}=O(n)/ (O(k)\times O(n-k))$.   These transformations were extensively studied in Gelfand's school,
  by S. Helgason, and other authors;
  see \cite {GGG, Hel90, Hel00, Hel, Ru2, Ru4, Ru03} and references therein. Further generalization  connects functions on two different Grassmannians by inclusion, namely,
  \be \label {pays}(R_{k,\ell} f) (\eta)= \int_{\xi\subset \eta} f(\xi)\, d_\eta \xi, \qquad \xi\in G_{n,k}, \quad \eta\in G_{n,\ell},\quad k<\ell,\ee
  where $d_\eta \xi$ denotes the probability measure on the Grassmann manifold $G_k (\eta)$ of all $k$-dimensional linear subspaces of $\eta$.
   Operators (\ref{pays}) are also known as Radon transforms for a pair of Grassmann manifolds. They    were studied
    by Petrov \cite{P1}, Gelfand and his collaborators \cite{GGS, GGR}, Grinberg \cite{Gri}, Kakehi \cite{Ka},  Grinberg and Rubin \cite{GR},  Zhang \cite{Zh1, Zh2}.  Affine versions of (\ref{pays}) were considered in \cite {GK1,GK2,  Ru04, Shi}.
    
    In spite of the elegance and ingenuity of the inversion methods  in these  publications, the resulting formulas are pretty involved.  It is a challenging problem to find new simple inversion formulas and develop a theory which is parallel to that for  the unit sphere.  The present paper is devoted to this problem. The main idea of our approach is to treat operators (\ref{pays}) as members of  the analytic family of the higher-rank cosine transforms, which will be  introduced below.

   {\bf 3.}     The name {\it cosine transform}  was  given by  Erwin Lutwak
   \cite[p. 385]{Lu} in 1990 to the integral operator
\be\label{t11}(\frC
f)(u)=\int_{S^{n-1}} f(v) |u \cdot v| \, dv, \qquad u \in
S^{n-1}.\ee
Since then, this term is widely used in  integral and convex geometry (in parallel with its traditional meaning in the Fourier analysis).
Operator (\ref{t11}) and its generalization with the  kernel $ |u \cdot v|^\lam$ were studied without naming by many authors  in geometry and analysis since
Blaschke \cite{Bla},   Levy \cite{Lev}, Aleksandrov \cite{Al}; see \cite{Ga, K97, Ru03}
for references.
A remarkable fact, that
amounts to the results of Gelfand-Shapiro \cite{GSha} and Semyanistyi
\cite{Se}, is that cosine transforms   are
restrictions to $S^{n-1}$ of the Fourier transforms of  homogeneous
 distributions on $\bbr^n$.  Specifically, if we set \[ (\frC ^{\a} f)(u)= \int_{S^{n-1}} f(v) |u
\cdot v| ^{\a -1} \, dv,\quad (\F\phi)(y)=\int_{\bbr^n} \phi (x) e^{ix\cdot y}\, dx,\]   $f\in L^1(S^{n-1})$,  $\phi\in S(\bbr^{n})$, then
 \bea\label {hrya}&&\frac{1}{\Gam (\a/2)}\int_{\bbr^n}  \frac{(\frC ^{\a}f)(y/|y|)}{|y|^{1-\a}}\, \overline{(\F\phi)(y)}\, dy\\
  &&\quad =\frac{c}{\Gam ((1\!-\!\a)/2)}\int_{\bbr^n}  \frac{f(x/|x|)}{|x|^{n+\a-1}}\, \overline{\phi(x)}\, dx, \qquad c=2^{n+\a}\pi^{n-1/2}.\nonumber\eea
Integrals in this equality converge absolutely when $0<Re \, \a <1$ and extend by analyticity to all $\a\in \bbc$.
 An important  observation due to Semyanistyi is that the Funk transform and its inverse are members of the  analytic family
  of suitably normalized cosine transforms $\frC ^{\a}$. This result was extended in \cite {Ru4} to the  Funk-Radon transforms over totally geodesic submanifolds of arbitrary dimension. The corresponding  cosine transforms have found  application in convex geometry \cite {Ru7, Ru10a, RZ}. Some ideas from \cite{Se} were rediscovered by Koldobsky, who found remarkable application of the relevant Fourier transform technique to geometric problems; see \cite {K05} and references therein.

 {\bf 4.}  In the last two decades a considerable attention  was attracted to
generalization of the cosine transform for functions on the Stiefel and Grassmann manifolds. The impetus
 was given in stochastic geometry  by Matheron \cite[p. 189]{Mat},
 who conjectured  that the higher-rank analogue of (\ref{t11}) (the precise definition is given later)
  is injective. Matheron's conjecture was disproved  by Goodey and Howard
\cite{GH1}, using topological results of Gluck and  Warner \cite{GW}. A self-contained Fourier analytic proof,  versus \cite{GH1}, was suggested by   Ournycheva  and  Rubin \cite {OR4, OR3}.
Goodey and Zhang \cite{GZ} applied higher-rank cosine transforms to the study of lower dimensional projections of convex bodies;  see also \cite{Goo, Spo}.
Interesting connections to  group representations can be found in the papers by  Alesker and Bernstein
\cite{AB}, Alesker \cite{A}, and Zhang \cite{Zh2}. One should also mention  fundamental publications   by Blind, Herz, Faraut, Khekalo, Ra\"{\i}s, Petrov, Ricci and Stein, Stein, and others. They are devoted to analysis of homogeneous distributions on matrix spaces; see  \cite {Bli1, Bli2, FK, Herz, Kh1a, Kh2,  OR4, OR3, P2, Rai, RS, St2} and references therein.

 \subsection{Plan of the paper and main results} Below we give  a brief account of main results and  ideas of the paper.  Further details and more results can be found in respective sections.
 Section 2 contains  preliminaries. We recall  basic facts about matrix spaces,
  Radon transforms over matrix planes, Riesz distributions, and  the composite power function associated to the cone $\Omega$ of positive definite symmetric $m\times m$ matrices.

In Section 3 we define the higher-rank Funk transform $F_{m,k}$ as an operator (\ref{la3v})
that sends functions on the Stiefel manifold $\vnm$ of orthonormal $m$-frames in $\bbr^n$ to functions on  $\vnk$, where $m$ and $k$ may be different.  We also define the dual transform $\fd$ and establish connection between our transforms and Radon transforms  (\ref{pays}).
\begin {theorem} \label{j7652} Let $1\le m \le k\le n-m$.
The mapping $f \to F_{m,k} f$ has a kernel $$\ker F_{m,k}=\{f \in L^1(\vnm): \; \int_{O(m)} f(v\gam)\,d\gam=0 \;  a.e.\}.$$
\end{theorem}

This theorem generalizes the well-known result in \cite{Hel}, where the case $m=1$ is considered for $C^\infty$ functions.

In Section 4 we introduce  analytic families of non-normalized  cosine  and sine  transforms for a pair of Stiefel manifolds; see
(\ref{0mby})-(\ref{c0mbyd}). Necessary and sufficient conditions are obtained for these transforms  to be represented by absolutely convergent integrals. Useful composition formulas, which generalize to  $m>1$ the corresponding equalities for $m=1$ in \cite{Ru4, Ru7},
are derived.  One of them is
\be \label {kliu}\fd F_{m, k}f\!=\!
\tilde c\,Q^{n-k-m} f,
\ee
where $Q^{\a} f$, $\a=n-k-m$, is the sine transform of $f$ and a constant $\tilde c$ is explicitly evaluated. This formula  is well-known for Radon transforms of different kinds, where the operator on the right-hand side stands for the relevant version of the Riesz potential;
see, e.g.,  \cite{Rad, Fug, Hel, Ru10}. In the higher-rank case an analogue of the composition $\fd F_{m, k}$ played a crucial role in \cite {Gri}, however, it was not explicitly computed  and  presented only in the spectral form on highest weight vectors.

Section 5 is devoted to the Fourier analysis and analytic continuation of the cosine transform. Given $\lv=(\lam_1,\ldots, \lam_m)\in\bbc^m$, we introduce {\it the
 composite cosine transform}
\be\label{tnfin}(T_{k,m}^{\lv} \vp)(v)=\int_{\vnk} \vp(u) \, (v'uu'v)^{\lv} \,
d_*u, \qquad v\in\vnm,\ee
 where $(\cdot)^{\lv}$ denotes the
 power function of the cone $\Omega$,   ``$\,{}'\,$'' stands for the transposed matrix, and integration is performed against the invariant probability measure on $\vnk$.  Injectivity of such operators in the case $k=m$ was studied in \cite{ OR4, OR3}.
  For (\ref{tnfin}) and for the dual Funk transform $\fd$ we obtain an analogue of the Fourier functional equation (\ref{hrya}); see Theorems \ref{gen} and \ref{nkild}.
    An important new feature is that a function $f$ on the right-hand side must be replaced by a certain {\it complementary Radon transform} of $f$. The latter boils down to the identity operator, when $k=m$.
The Fourier functional equation for the operator (\ref{tnfin}) provides complete information about  meromorphic structure of  distributions of the form
\be\label {ui91}  \a \to (T^\a_{m,k} f, \om), \qquad \om \in C^\infty (\vnk),\quad f\in L^1(\vnm), \ee
where $T^\a_{m,k}$ stands for the cosine or sine transform under consideration.
    We conjecture that if $f\in C^\infty (\vnm)$ then the polar set of the distribution (\ref{ui91}) coincides with that of its pointwise counterpart
    $$\a \to (T^\a_{m,k} f)(u) \quad \mbox{\rm for each } \quad  u\in\vmk.$$ To the best of our knowledge, a  proof of this fact represents an open problem.

In Section 6 we define  normalized  versions of the cosine and sine transforms. One of the main results of this section is Theorem \ref{cr24}, according to which an
integrable   right
$O(m)$-invariant function $f$ on $\vnm$, can be reconstructed in the sense of distributions from its Funk transform $\vp=F_{m,k} f$  by the formula
\be \label {forq1in}  \underset
{\a=m +k-n}{a.c.} (\ncd0 \vp, \om)\!=\!\varkappa_k (f,\om), \qquad \om \in C^\infty
(\vnm),\ee
where
$$(\ncd0 \vp)(v)=\del_{n,m,k} (\a) \int_{\vnk} \!\!\vp(u)\, (\det(v'uu'v))^{(\a-k)/2} \, d_*u, \quad v\in\vnm,$$
is  the normalized  dual cosine transform and the constant $\varkappa_k$ is explicitly evaluated.  Thus, an inverse Funk transform is actually
 a member of the analytic family of suitably normalized dual cosine transforms,
 and all possible inversion formulas for the Radon transform  (\ref{pays}) on Grassmannians can be regarded as different realizations of the analytic continuation in (\ref{forq1in})\footnote
 {This statement does not work for the restricted Radon transform  as in \cite{GGS, GGR}, because the latter is, in fact, another operator.}.  Since $(f,\om)=0 \;\forall\om \in C^\infty
(\vnm)$ implies  $f=0$ a.e., the validity of Theorem \ref{j7652} follows; see also Remark \ref{78hy}.

Similar inversion results are obtained for the cosine transforms. An interesting observation is that, {\it if $k=m$, then the Funk transform, the cosine transforms, and their inverses are (up to normalization)  members of the same analytic family of operators.} This fact was known before only for $m=1$ and established using  decomposition in spherical harmonics.

In Section 7 we suggest new realization of the inverse sine and Funk transforms in terms of powers of the Cayley-Laplace  operator $\Del=\det (\d/\d x_{i,j})$ in the space $\Ma$  of $n\times m$ real matrices.  These powers are associated with the Riesz potential. The resulting formulas differ in principle from those in \cite{GGR, Gri, GR, Ka, P1}. They look much simpler, however, must be interpreted in the sense of distributions. For instance, if we write $x\in \Ma$ in polar coordinates
 $x=vr^{1/2}$, $r\in \Omega$, $v \in \vnm$, and  denote
$(E_\lam f)(x)=(\det(r))^{\lam/2} f(v)$, then, in the case of $n-k-m$  even, we  have
\[f(v)=c_{m,k}\,(\Del^{\ell} \,E_{-k}\fd \vp)(v), \qquad \vp=F_{m,k}f, \quad \ell=(n-k-m)/2,\]
where the constant $c_{m,k}$ is explicitly evaluated.  If $n-k-m$ is an odd number, then fractional powers of  $\Del$ are implemented, and the resulting inversion formula is non-local; see Theorems \ref{kreven}  and  \ref{cr2t}.
In the case $m=1$ inversion formulas of this kind were suggested by Semyanistyi \cite{Se} and used in \cite {Str81, Ru4}. The reasoning from those papers is unapplicable when $m>1$. To get around  this difficulty, we apply  the relevant tools of harmonic analysis and several complex  variables.

We conclude the paper by Appendix containing evaluation of an auxiliary integral and some comments on the celebrated paper \cite{GGR} by Gelfand,  Graev, and  Rosu. The latter is  devoted to the Radon transform (\ref{pays}) for a pair of Grassmannians. A keen reader can recognize  our higher-rank cosine transform in  \cite[pp. 367, 368]{GGR}.  However, some important details in that paper are skipped. We  reproduce the relevant calculations  and explain basic difficulties.

{\bf Acknowledgements.} I am indebted to many people for useful and pleasant discussions related to this paper. My special thanks go to  Professors Tomoyuki Kakehi, Gestur \`Olafsson, Elena Ournycheva, Angela Pasquale, Fulvio Ricci,   Genkai Zhang. I am grateful  to Susanna Dann who carefully read the original Funk papers. Her help was very valuable.

\section{Preliminaries}  In this section we establish our notation and recall some basic facts; see \cite{OR1, OR4, OR3} for more details.

\subsection{Notation and conventions} Given a square matrix $a$,  $|a|$ stands for the absolute value of
  $\det(a)$.  We use  standard
 notation $O(n)$   and $SO(n)$ for the orthogonal group and the
 special orthogonal group of $\bbr^{n}$, respectively, with the  normalized
 invariant measure of total mass 1. The abbreviation ``$a.c.$'' denotes
 analytic continuation.

 Let  $\frM_{n,m}\sim\bbr^{nm}$ be  the
space of real matrices $x=(x_{i,j})$ having $n$ rows and $m$
 columns;  $dx=\prod^{n}_{i=1}\prod^{m}_{j=1} dx_{i,j}$;
   $x'$ is  the transpose of  $x$, $|x|_m=\det
(x'x)^{1/2}$, $I_m$
   is the identity $m \times m$
  matrix, and $0$ stands for zero entries.

  The Fourier transform  of a
function $\vp\in L^1(\frM_{n,m})$ is defined by \be\label{ft}
\hat\vp (y)=(\F\vp)(y)=\int_{\frM_{n,m}} e^{{\rm tr(iy'x)}} \vp
(x)\, dx,\qquad y\in\frM_{n,m} \; .\ee The relevant Parseval
equality
 has the form \be\label{pars} (\hat \vp, \hat \phi)=(2\pi)^{nm} \, (\vp, \phi),
\qquad (\vp, \phi)=\int_{\frM_{n,m}} \vp(x) \overline{\phi(x)} \,
dx.\ee This equality with $\phi$ in the Schwartz class $S(\frM_{n,m})$ of rapidly decreasing smooth functions
is used to define
the Fourier transform  of the corresponding  distributions.
If $\phi\in S(\frM_{n,m})$ and $\check \phi=\F^{-1}\phi$ is the
inverse Fourier transform of $\phi$, then, clearly,
\be \label{check} \check \phi (x)=(2\pi)^{-nm}\hat\phi (-x), \qquad
[\hat\phi]^\wedge (x)=(2\pi)^{nm}\phi (-x).\ee

Let $\S_m \sim \bbr^{m(m+1)/2}$  be the space of $m \times m$ real
symmetric matrices $s=(s_{i,j})$; $ds=\prod_{i \le j} ds_{i,j}$. We
denote by  $\Omega=\p$   the cone of positive definite matrices in $\S_m$;
$\bar\Omega$  is the closure of $\Omega$.  For $r\in\Omega$ ($r\in\bar\Omega$), we write
$r>0$ ($r\geq 0$). Given
 $a$ and  $b$ in  $S_m$, the inequality $a >b$  means $a - b \in
 \Omega$ and  the
symbol $\int_a^b f(s) ds$ denotes
 the integral over the set $(a +\Omega)\cap (b -\Omega)$.
 The group $G=GL(m,\bbr)$  of
 real non-singular $m \times m$ matrices $g$ acts transitively on $\Omega$
  by the rule $r \to g rg'$.  The corresponding $G$-invariant
 measure on $\Omega$ is \cite[p.
  18]{T}  \be\label{2.1}
  d_{*} r = |r|^{-d} dr, \qquad  d= (m+1)/2. \ee
If $T_m$ is the group  of upper triangular $m \times m$
  real matrices $t=(t_{i,j})$ with positive
diagonal elements, then each $r \in \Omega$ has a   unique
representation $r=t't$.

The  Siegel gamma  function of $\Omega$
 is defined by
\be\label{2.4}
 \gm (\a)\!=\!\int_{\Omega} \exp(-\tr (r)) |r|^{\a } d_*r
 =\pi^{m(m-1)/4}\prod\limits_{j=0}^{m-1} \Gam (\a\!-\! j/2).  \ee
 This integral is absolutely convergent if and only if $Re
 \, \a >d-1=(m-1)/2$, and extends  meromorphically  with
 the  polar set  \be\label {09k}\{(m-1-j)/2: j=0,1,2,\ldots\};\ee see
\cite{Gi}, \cite{FK},
 \cite{T}. For the corresponding beta integral we have
\be\label{2.6d}
 \int_a^b |r-a|^{\a -d} |b-r|^{\beta -d} dr= B_m (\a ,\b)\, |b-a|^{\a
 +\b-d}, \ee \[
 B_m (\a ,\b)=\frac{\gm (\a)\gm (\b)}{\gm (\a+\b)}, \qquad  Re
 \, \a >d-1, \; Re
 \, \b >d-1.\]

For $n\geq m$, let $\vnm= \{v \in \frM_{n,m}: v'v=I_m \}$
 be  the Stiefel manifold  of orthonormal $m$-frames in $\bbr^n$. This is a homogeneous space
with respect to the action $\vnm \ni v\to \gam v $, $\gam\in O(n)$,
so that   $\vnm=O(n)/O(n-m)$.
We  fix a measure $dv$ on $\vnm$, which is left $O(n)$-invariant,
right $O(m)$-invariant, and
  normalized by \be\label{2.16} \sigma_{n,m}
 \equiv \int_{\vnm} dv = \frac {2^m \pi^{nm/2}} {\gm
 (n/2)}, \ee
\cite[p. 70]{Mu}. The notation $d_\ast v=\sig^{-1}_{n,m} dv$ is used for the corresponding probability measure.

We denote by $G_{n,m}$ the Grassmann manifold of $m$-dimensional linear subspaces $\xi$ of $\bbr^n$ equipped with the $O(n)$-invariant probability measure $d_*\xi$. Every right $O(m)$-invariant function $f (v)$  on $V_{n,m}$ can be identified with a function $\tilde f (\xi)$ by the formula
$\tilde f (\{v\})=f(v)$, $ \{v\}=v\bbr^m\in G_{n,m}$,  so that $\int_{G_{n,m}} \tilde f (\xi)d_*\xi=\int_{V_{n,m}} f (v)d_*v$. Another identification is also possible, namely, $\stackrel{*}{f} (\{v\}^\perp)=f(v)$, $ \{v\}^\perp\in G_{n,n-m}$.

We will be dealing with several  coordinate systems on $\frM_{n,m}$ and $\vnm$.

\begin{lemma}\label{l2.3} {\rm (The polar decomposition).} Let $x \in \frM_{n,m}, \; n \ge m$. If  $\rank (x)=
m$, then \be\label{pol} x=vr^{1/2}, \qquad v \in \vnm,   \qquad
r=x'x \in\p,\ee and $dx=2^{-m} |r|^{(n-m-1)/2} dr dv$.
\end{lemma}
For this statement see, e.g.,  \cite{Herz},  \cite{Mu},
\cite{FK}. Decomposition (\ref{pol}) (but with $r\in\bar\Omega$)  is valid
for {\it any} matrix $x\in\Ma$, cf. \cite[p. 589]{Mu}.

\begin{lemma}\label{sph} {\rm (\cite{P2}, \cite{Ru10})} If
$x \in \Ma, \; \rank (x)=m,\;  n \ge m$, then
  \[
x=vt, \qquad v \in \vnm,   \qquad t \in T_m,\] and
$$
dx=\prod\limits_{j=1}^m t_{j,j}^{n-j} \,dt_{j,j}\,dt_*dv, \qquad
dt_*=\prod\limits_{i<j} dt_{i,j}.
$$
\end{lemma}

\begin{lemma}\label{l2.1} {\rm (The bi-Stiefel decomposition).} \ Let $k$,
$m$, and $n$  be positive integers satisfying $$1 \le k  \le n-1,
\qquad 1 \le m  \le n-1, \qquad k+ m \le n.$$

\noindent {\rm (i)} \ Almost all matrices $v \in \vnm$ can be
represented in the form
 \be\label{herz1}
  v= \left[\begin{array} {cc} a \\ u(I_m -a'a)^{1/2}
 \end{array} \right], \qquad a\in \frM_{k, m}, \quad u \in V_{n-k,m},
 \ee
  so that
   \be\label{2.10}
   \intl_{\vnm} f(v) dv
 =\intl_{0<a'a<I_m}d\mu(a)
 \intl_{ V_{n- k, m}} f\left(\left[\begin{array} {cc} a \\ u(I_m -a'a)^{1/2} \end{array}
 \right]\right) \, du,
\ee
 $$
 d\mu(a)=|I_m-a'a|^\delta da,\quad \del=(n-k)/2-d, \quad  d=(m+1)/2.
 $$

 \noindent  {\rm (ii)} \ If, moreover, $k \ge m$, then
\be\label{2.11}
 \intl_{\vnm} f(v) dv
 =\intl_0^{I_m} d\nu(r) \intl_{ V_{k, m}}dw
 \intl_{ V_{n-k, m}} f\left(\left[\begin{array} {cc} wr^{1/2} \\ u(I_m -r)^{1/2} \end{array}
 \right]\right) \ du,
\ee $$
 d\nu(r)=2^{-m} |r|^\gam |I_m -r|^\del dr, \qquad \gam
=k/2-d.
 $$
\end{lemma}

 For  $k=m$, this statement is due to  \cite[p. 495]{Herz}.
 The proof of Herz was extended in \cite{GR} to all $k+m \leq n$ and simplified in \cite{Zh1}; see also \cite{OlR}.

\begin{lemma} \label {hbi1} Let $u \in \vnk, \; v \in V_{n,m}; \; 1 \le k,m \le n$. If $f$ is a
function of $m \times k$ matrices, then \be \label {mlo}
 \intl_{\vnk} f(v'u) \,d_*u =
\intl_{V_{n,m}} f(v'u) \,d_*v. \ee
\end{lemma}
\begin{proof} We should observe that formally
 the left-hand side is a function of $v$, while the right-hand side
 is a function of $u$. In fact, both are constant. To prove
 (\ref{mlo}),
 let $G=O(n), \; g \in G, \; g_1=g^{-1}$. The left-hand side is
\[
 \int_G  f(v'gu) \,dg  =\int_G  f((g_1v)'u)\, dg_1,
\]
which equals the right-hand side.
\end{proof}

The following statement is a particular case of Lemma 2.5 from \cite{GR}.

\begin{lemma}\label {25gr} Let $A \in  \frM_{k,\ell},\; S=A'A \in \pl, \;  \ell \le m<k, \; m+\ell \le k$,
\[ \del =(\ell +1)/2, \qquad c=2^{-\ell} \sig_{k-m,\ell}
 \sig_{m,\ell}/\sig_{k,\ell}.\]
 Then
\bea &&\int_{V_{k,m}} \!f(A'\om)\, d_*\om\\&&=\!c\, |S|^{\del
-k/2}\int_0^S |S\!-\!s|^{(k-m)/2 -\del}\, |s|^{m/2 -\del}
ds\int_{V_{m,\ell}} \!\! f(s^{1/2}\theta)\, d_*\theta.\nonumber\eea
\end{lemma}

\subsection{The composite power function} \label {secpf} Let  $\Omega$ be the cone
of positive definite symmetric $m\times m$ matrices. Given $r=(r_{i,j})\in\Omega$, let $\Del_0(r)=1$, $\Del_1(r)=r_{1,1}$,
$\Del_2(r)$, $\ldots$, $\Del_m(r)=|r|$ be the corresponding
principal minors, which are strictly positive. For
$\lv=(\lam_1,\dots,\lam_m)\in\bbc^m$, the composite power function
of the cone  $\Omega$  is defined by
 \bea\label{pf} r^{\lv}&=&\prod\limits_{i=1}^m
 \left[\frac{\Del_i (r)}{\Del_{i-1} (r)}\right]^{\lam_i/2}\\&=&
  \Del_1 (r)^{\frac{\lam_1 -\lam_2}{2}}  \ldots \Del_{m-1}
(r)^{\frac{\lam_{m-1} -\lam_m}{2}} \Del_m (r)^{\frac{\lam_m}{2}}.\nonumber
\eea
  If $r=t't, \; t
=(t_{i,j})\in T_m$, then \be\label {mku5} r^{{\lv}}=\prod_{j=1}^m
 t_{j,j}^{\lam_j}.\ee This implies the following equalities:
\bea\label{pr1} r^{{\lv+\mv}}&=&r^{{\lv}}\;r^{{\mv}}, \quad
r^{{\lv+\av_0}}=r^{\lv}|r|^{\a/2}, \quad \av_0=(\a,\dots,\a);
\\ \label{pr6} \; (t'rt)^{{\lv}}&=&(t't)^{\lv}\;r^{{\lv}},\quad t\in
T_m. \eea The reverses of $\lv=(\lam_1,\dots,\lam_m)$ and
$r=(r_{i,j})\in\Omega$ are defined by \be\label{ommp} \lv_\ast=(\lam_m,\dots,\lam_1);
\qquad r_\ast =\om r\om,\qquad  \om=\left[\begin{array}{ccccc}
0 & {}   & {}    & 1 \\
                              {} & {}  & {.}    & {} \\
                              {} & {.}   & {}   & {} \\
                               1 & {}   & {}    & 0
\end{array} \right],\ee
so that $$ (\lv_\ast)_j=\lam_{m-j+1}, \qquad
(r_\ast)_{i,j}=r_{m-i+1, m-j+1}.$$ We have
 \be\label{pr4}
 r^{{\lv_\ast}}=(r^{-1})_\ast^{-\lv},\qquad
(r^{-1})^{{\lv}}=r_\ast^{{-\lv_\ast}}.\ee
The  relevant gamma function  is defined by \be\label{gf}
\Gam_{\Omega}  (\lv) =\intl_{\Omega} r^{\lv} e^{-{\rm tr} (r)} d_{*}
r=\pi^{m(m-1)/4}\prod\limits_{j=1}^{m} \Gam ((\lam_j- j+1)/2);\ee
 see, e.g., \cite[p. 123]{FK}. The integral in (\ref{gf}) converges
 absolutely if and only if $Re \, \lam_j>j-1$ for all $j=1,\dots,m$, and extends
 meromorphically to all $\lam \in \bbc^m$.
 The following relation holds:
 \be\label{eq11} \intl_{\Omega}
r^{\lv} e^{-{\rm tr} (rs)} d_{*} r= \Gam_{\Omega} (\lv) \,
s_\ast^{{-\lv_\ast}}, \qquad s\in\Omega.\ee
If $\lv_0=(\lam,\dots,\lam)\;(\in \bbc^m)$, then (cf. (\ref{2.4}))
 \be\label{det} r^{{\lv_0}}=|r|^{\lam/2},\qquad \Gam_{\Omega}  (\lv_0)=\gm(\lam/2).\ee

 \subsection{The Radon transform on the space of matrices}
  The main
references for this topic are  \cite{OR1}, \cite{OR2}, \cite{OR5},
\cite{P2}, \cite{Sh1}, \cite{Sh2}. Close results can be found in \cite{GK3, Gra}.
 We fix  positive integers $k,n$,
and $m$, $0<k<n$, and let $\vnk$ be the Stiefel manifold of
orthonormal $k$-frames in $\bbr^n$. For $\; u\in\vnk$ and
$t\in\frM_{k,m}$, the linear manifold \be\label{plane} \tau=
\tau(u,t)=\{x\in\Ma:\, u'x=t\} \ee
 is called a {\it  matrix $(n-k)$-plane} in $\Ma$.  We denote by  $\Gr$ the
 set of all such
 planes and
consider the  Radon transform
\[ (\R_k
f) (\tau)=\int_{x \in \tau} f(x),\] that sends a function $f$ on
$\frM_{n,m}$ to a function $\R_k f$  on $\Gr$.  Precise meaning of
this integral is
 the following:
\be\label{4.9} (\R_k f)(u,t)=\int_{\frM_{n-k,m}} f\left(g_u
\left[\begin{array} {c} \om \\t
\end{array} \right]\right)d\om,
\ee where $ g_u \in SO(n)$ is a rotation satisfying
\be\label{4.24}
g_u u_0=u, \qquad u_0=\left[\begin{array} {c}  0 \\
I_{k} \end{array} \right] \in \vnk. \ee

 The next statement is a matrix generalization of
the  projection-slice theorem. It links together the
Fourier transform (\ref{ft}) and the Radon transform (\ref{4.9}). In
the case $m=1$, this theorem can be found in \cite[p. 11]{Na} ($k=1$)
and \cite[p. 283]{Ke} (any $1\le k<n$).
 \begin{theorem}\label{CST} {\rm(\cite{OR6})}
 For  $f\in L^1(\Ma)$ and  $1\le m\le k$,
\be\label{4.20} (\F f)(\xi b)=[\tilde \F (\R_k f)(u,\cdot)](b),
\quad u\in\vnk, \quad b\in\frM_{k,m} \ee where $\tilde \F$ stands
for the Fourier transform  on  $\frM_{k,m}$.
 \end{theorem}

 \subsection{Riesz potentials and the Cayley-Laplace operator} \label {se4} The {\it Riesz distribution} $h_\a$ on $\Ma$ is defined by
 \be\label{rd} (h_\a, f)= a.c. \; \frac{1}{\gam_{n,m} (\a)}
  \int_{\Ma} |x|^{\a-n}_m f(x) dx,\qquad f\in\S(\Ma),
  \ee
 \be\label{gam} \gam_{n,m} (\a)\!=\!\frac{2^{\a m} \, \pi^{nm/2}\, \Gam_m
(\a/2)}{\Gam_m ((n\!-\!\a)/2)},\qquad \a\!\neq \!n\!-\!m\!+\!1, \,
n\!-\!m\!+\!2, \ldots . \ee  For $Re \, \a>m-1$, the distribution
$h_\a$ is
 regular and agrees with the ordinary function $h_\a(x)=|x|^{\a-n}_m
 /\gam_{n,m}(\a)$.
 The {\it Riesz potential} of a function $f\in\S(\Ma)$
 is defined by \be\label{rie-z} (I^\a
f)(x)=(h_\a, f_x), \qquad f_x(\cdot)= f(x-\cdot).\ee For $Re \,
\a>m-1$, $ \a\neq n-m+1, \, n-m+2, \ldots \;$, (\ref{rie-z}) is
represented in the classical form by the absolutely convergent
integral
 \be\label{rie} (I^\a f)(x)=\frac{1}{\gam_{n,m} (\a)} \int_{\Ma}
f(x-y) |y|^{\a-n}_m dy.\ee This integral operator is well known in
the rank-one case $m=1$ \cite {Ru96, SKM, Sa2, St1}.

 The {\it Cayley-Laplace operator} $\Del$  on the space $\Ma$  is  defined by \be\label{K-L} \Del=\det(\d
'\d). \ee Here $\partial$ is an $n\times m$  matrix whose entries
are partial derivatives $\d/\d x_{i,j}$. In the Fourier transform
terms, the action of $\Del$ represents a multiplication by the
polynomial $(-1)^m |y|_m^2$.

\begin{theorem}\label{m76}\cite [Theorem 5.2]{Ru10} {}\hfil

\noindent Let $f\in\S(\Ma)$,  $ \;\a\in \bbc$,  $ \;\a\neq n-m+1, \,
n-m+2, \ldots \,$. Then
\be\label{fou} (h_\a, f)=(2\pi)^{-nm}(|y|_m^{-\a},(\F f)(y)), \ee
\be \label{Dkh}(-1)^{mk}\Del^k
h_{\a+2k}\!=\! h_{\a}, \quad I^{-2k}\!=\!(-1)^{mk}\Del^k; \quad  k\!=\!0,1,2,\dots \,.\ee
\end{theorem}

Thus, one can formally write \be  \label
{bvv}I^\a= [(-1)^m \Del ]^{-\a/2}.\ee

 The next statement generalizes (\ref{fou}) to the case of composite power functions.

 \begin{lemma} Let $\phi\in\S(\frM_{n,m})$. Then for all $\lv\in \bbc^m$,
 \be \label{eq8}  \intl_{\frM_{n,m}}  \frac{(y' y)^{
\lv}}{\Gam_{\Omega} (\lv+\nv)} \;\overline{(\F\phi)(y)}\, dy=c_\lv\,
\intl_{\frM_{n,m}}\frac{(x' x)_*^{-\lv_* -\nv}}{\Gam_{\Omega} (-\lv_*)} \;\overline{\phi(x)}\, dx, \ee
where $c_\lv=2^{nm+|\lv|}\pi^{nm/2}, \; |\lv|=\lam_1+\ldots +\lam_m$.
\end {lemma}

  \begin{remark} \label {bnuk} {\rm Both sides of  (\ref{fou}) and (\ref{eq8}) are understood in the sense of analytic continuation.
  Integrals in (\ref{eq8}) converge simultaneously when $\lv=(\lam_1, \ldots , \lam_m)$ belongs to
  the set $\frL=\check \frL \cap  \tilde \frL$, where
 \bea &&\check \frL = \{\lv:\; Re\,\lam_j> j-n-1 \quad \mbox{\rm for each $j=1,\dots, m$} \}, \nonumber\\
 && \tilde \frL= \{\lv:\; Re\,\lam_j<j-m \quad \mbox{\rm for each $j=1,\dots, m$} \}.\nonumber\eea
 The diagonal  $\lam_1=\ldots=\lam_m=\lam$ does not belong to $\frL$.
 This explains essential difficulties  when one tries to prove  (\ref{fou}) directly as in \cite {Ru10}. Formula (\ref{eq8}) was established by Khekalo
\cite{Kh1a}, who extended    the  argument from \cite[Chapter III,
Sec. 3.4] {St1} to functions of matrix argument. Khekalo's proof was reproduced in \cite[p. 61]{OR4}.}
\end{remark}

\section{The Higher-Rank Funk Transform}
\setcounter{equation}{0}

\subsection{Definitions and duality}  The classical Funk transform  on the
unit sphere $S^{n-1}\subset \bbr^n$ is defined by
\be \label {7a3v}(Ff) (u)=\int_{\{v\in S^{n-1}: \, u\cdot v=0\}} f(v)\,d_u v,
\qquad u\in S^{n-1};\ee
see, e.g., \cite {GGG, Hel}. We suggest the following  generalization of (\ref{7a3v}), in which $u\in V_{n, k}$ and $v\in V_{n, m}$ are elements of respective  Stiefel manifolds, $ 1\le k,m\le n-1$.
The {\it higher-rank Funk transform } sends a function $f$ on $V_{n, m}$ to a function $F_{m,k} f$ on $V_{n, k}$ by the formula
 \be \label {la3v}(F_{m,k} f) (u)=\int_{\{v\in V_{n, m}: \, u'v=0\}} f(v)\,d_u v,
\qquad u\!\in\! V_{n, k}.\ee
The corresponding dual transform
 \be \label {la3vd}(\fd \vp) (v)=\int_{\{u\in V_{n, k}: \, v'u=0\}} \vp(u)\,d_v u,
\qquad v\!\in\! V_{n, m},\ee
 acts in the opposite direction. The condition $u'v=0$ means that subspaces $u\bbr^k \in \gnk$ and $v\bbr^m \in \gnm$ are mutually orthogonal. Hence,    necessarily,  $k+m\le n$. The case $k=m$, when both $f$ and its Funk transform live on the same manifold, is of particular importance and coincides with (\ref{7a3v}) when $k=m=1$. We denote $F_m= F_{m,m}$.

 To give  the new  transforms precise meaning,  we set $G=O(n)$,
 \bea
K_0&=&\left \{ \tau \in G : \tau = \left[\begin{array} {cc} \gam & 0
\\ 0 & I_{k}
\end{array} \right], \quad \gam  \in
O(n-k) \right \},\\
\label {klop7} \check K_0&=&\left \{ \rho \in G : \rho = \left[\begin{array} {cc}
\del & 0
\\ 0 & I_m \end{array} \right], \quad  \del
\in O(n-m) \right \}, \eea
\be\label {klop8} u_0= \left[\begin{array} {c}  0 \\  I_{k} \end{array} \right],
 \quad \check{u}_0=  \left[\begin{array} {c}  I_k \\  0 \end{array}
\right]; \qquad  v_0= \left[\begin{array} {c}  0 \\  I_{m} \end{array} \right],
 \quad \check{v}_0= \left[\begin{array} {c}  I_m \\  0 \end{array}
\right ],\ee

\[ u_0,  \check{u}_0 \in  \vnk; \qquad  v_0,  \check{v}_0 \in \vnm.\]
Then (\ref{la3v}) and (\ref{la3vd}) can be explicitly written as \be\label{876a}
(F_{m, k}f) (u)=\int_{V_{n-k,m}} \!\!\!\!\!f\left(g_u
\left[\begin{array} {c} \om
\\0
\end{array} \right]\right)\,d_*\om\!=\!\int_{K_0} \!\!f(g_u \t
\check{v}_0)\,d\t,
 \ee
 \be \label{876ab}(\fd\vp) (v)=\int_{V_{n-m,k}} \!\!\!\!\!\vp\left(g_v
\left[\begin{array} {c} \theta
\\0
\end{array} \right]\right)\,d_*\theta\!=\!\int_{\check K_0} \!\!\vp(g_v \rho
\check{u}_0)\,d\rho,
 \ee
where  $ g_u$ and  $ g_v$ are orthogonal transformations satisfying $g_u u_0=u$
 and $g_v v_0=v$, respectively.

\begin{remark}\label {78hy}{\rm Since the measure $d_*\om$ is right $O(m)$-invariant, then, for all $u\in \vnk$,
\be (F_{m, k}f) (u)=(F_{m, k}\tilde f) (u), \qquad \tilde f (v)=\int_{O(m)} f(v\gam)\, d\gam\ee
(similarly for $(\fd\vp) (v)$). Hence, the set of all $f \in L^1(\vnm)$, for which $\tilde f (v)=0$  a.e., is a subset of  $\ker F_{m,k}$; cf. Theorem \ref{j7652}.}
\end{remark}

\begin{lemma} Let $ 1\le k,m\le n-1$; $\,k+m\le n$. Then
\be\label{009a}\int_{V_{n, k}}(F_{m, k}f)(u)\, \vp(u)\,
d_*u=\int_{\vnm}f(v)\,(\fd \vp)(v)\, d_*v\ee provided that at
least one of these integrals is finite when $f$ and $\vp$ are
replaced by $|f|$ and $|\vp|$, respectively.

\end{lemma}
\begin{proof}

We write the left-hand side as
 \[
\int_G  (F_{m, k}f)(g u_0)\vp(g u_0)\, dg=\int_G \vp(g u_0)\,
dg\int_{K_0} \!\!f(g \t \check{v}_0)\,d\t.\] Let $\varkappa \in
O(n)$ be such that $u_0=\varkappa \check{u}_0$ and
 denote
\be\label {78v}\z=\left[\begin{array} {ccc}  0 & 0 & I_m\\
 0 & I_{n-m-k} & 0 \\
I_k & 0 & 0\end{array} \right]. \ee Keeping in mind that $\z v_0
=\check{v}_0$, $\z' u_0 =\check{u}_0$, and $\check{K}_0$ is a
stabilizer of $v_0$, we continue: \bea l.h.s&=&\int_G \vp (g
\varkappa \check{u}_0)\, dg \int_{\check K_0}\, d\rho \int_{K_0} f
(g \tau \z\rho' v_0)\,
 d\tau \nonumber\\
&=&\int_G f (\lam v_0) \,d \lam \int_{\check K_0} d\rho
\int_{K_0}\vp (\lam\rho \z'\t'\varkappa  \check{u}_0)
\, d\t\nonumber\\
&{}& \mbox {\rm (note that $\z'\t'\varkappa  \check{u}_0=\z'\t'u_0=\z' u_0=\check{u}_0$)}\nonumber\\
&=&\int_G f (\lam v_0) \,d \lam \int_{\check K_0} \vp (\lam
\rho\check{u}_0)\,
 d\rho\nonumber\\
&=&\int_{\vnm} f(v)\,(\fd \vp)(v)\, d_*v,\nonumber \eea as
desired.
\end{proof}

\begin{corollary}\label {laas} If $f \in L^1 (\vnm)$, then the Funk transform $(F_{m,
k}f)(u)$ exists as an absolutely convergent integral for almost all
$u\in \vnk$. Moreover,\[\int_{V_{n, k}}(F_{m, k}f)(u)\,
d_*u=\int_{\vnm}f(v)\, d_*v.\]
\end{corollary}

A similar statement holds for the dual transform $\fd \vp$.

\begin{example} \label {iikd}  Let $\vp (u)= |v_0' uu'v_0|^{(\a-k)/2}$,
$v_0= \left[\begin{array} {c}  0 \\  I_{m} \end{array} \right]$, $\;u\in V_{n, k}$.
If $1\le m\le k\le n-m, \; Re \,\a>m-1$, then
 \be \label {kja} (\fd\vp) (v)= c_\a \, |I_m-v_0'
vv'v_0|^{(\a-k)/2},\qquad v \in \vnm,\ee
\[
c_\a=\frac{\Gam_m((n-m)/2)}{\Gam_m(k/2)}\,\frac{\Gam_m(\a/2)}{\Gam_m((\a+n-m-k)/2)}.\]
\end{example}
\begin{proof} The condition $m \le k$ has a simple explanation: if $k<m$ then $|v_0' uu'v_0|\equiv 0$.
Let us prove (\ref{kja}). By (\ref{876ab}),
\[(\fd\vp) (v)=\int_{V_{n-m,k}} |z_\theta
z'_\theta|^{(\a-k)/2}\,d_*\theta, \qquad z_\theta= v'_0 g_v \left[\begin{array} {c} \theta
\\0
\end{array} \right].\]
We write
\be\label {nio} g'_v v_0=\left[ \begin{array} {c} A \\
 B \end{array} \right], \qquad A \in \frM_{n-m,m}, \quad  B \in
 \frM_{m,m},
\ee
 so that $z_\theta=A' \theta$. Then we represent $A$ in polar coordinates
 $A=wS^{1/2}$, $S=A'A\in \p, \; w\in V_{n-m,m}$. This gives
\[(\fd\vp)(v)=\int_{V_{n-m,k}}
|A'\theta\theta'A|^{(\a-k)/2}\,d_*\theta=c_\a\, |S|^{(\a-k)/2},\] where
\[
c_\a=\int_{V_{n-m,k}}
|w'\theta\theta'w|^{(\a-k)/2}\,d_*\theta\] can be computed using formula (\ref{mnv}) (with $n$ replaced by $n-m$ and $\lam$ by $\a-k$).
Let $\varkappa=\left[ \begin{array} {c} I_{n-m} \\
 0 \end{array} \right]\in V_{n,n-m}$. Then, by (\ref{nio}), $A=\varkappa' g'_v v_0$
 and \[ S=A'A= v'_0 g_v \varkappa \varkappa' g'_v v_0=v'_0 g_v (I_n -v_0 v'_0) g'_v v_0=I_m-v'_0 vv' v_0.\]
This gives the result.
\end{proof}

Example \ref{iikd} and duality (\ref{009a}) yield
\[ \int_{V_{n, k}}\!\!\!(F_{m, k}f)(u)\, |v'_0
uu'v_0|^{(\a-k)/2} \, d_*u\!=\! c_\a\,
\int_{\vnm}\!\!\!\!f(v)\,|I_m\!-\!v'_0
vv'v_0|^{(\a-k)/2}\,
d_*v.\]
 Owing to
$O(n)$-invariance, one can replace $v_0$ in this formula by
 an arbitrary $w \in V_{n, m}$. This gives the following statement.
\begin{lemma} \label{pmz}Let $w \in V_{n,
m}$, $ Re \,\a>m-1$,  \be \label {olpz}1\le m\le k\le n-m.\ee
  Then \be \int_{V_{n, k}}\!\!\!(F_{m, k}f)(u)\, |w'
uu'w|^{(\a-k)/2} \, d_*u\!=\! c_\a\, \int_{\vnm}\!\!\!f(v)\,|I_m\!-\!w'
vv'w|^{(\a-k)/2}\, d_*v,\nonumber\ee
 provided that  the integral on the right-hand side is absolutely convergent.
\end{lemma}

\subsection{Connection with Radon transforms on Grassmannians}
There is an intimate connection between the higher-rank Funk transform and the well-known
Radon transforms for a pair of Grassmann manifolds by inclusion. The latter were studied by several authors,
who used different methods; see, e.g., \cite{GGR, GGS, Gri, GR, Ka, P1, Zh1, Zh2}.

 Suppose that $f(v)$ and $\vp (u)$ are right $O(m)$-invariant  and $O(k)$-invariant functions on $\vnm$ and $\vnk$, respectively.
We define the corresponding functions on Grassmannians by setting
\be \tilde f(\xi)= \stackrel{*}{f}(\t)=f(v), \qquad\xi\!=\!\{v\}\!\in \!G_{n,m}, \quad \t\!=\!\{v\}^\perp\!\in \!G_{n,n-m},\ee
\be \tilde \vp(\z)= \stackrel{*}{\vp}(\eta)=\vp (u), \qquad\z\!=\!\{u\}\!\in \!G_{n,k}, \quad \eta\!=\!\{u\}^\perp\!\in\! G_{n,n-k},\ee
 and consider the Radon transforms
\be\label {i43} (R_{m,n-k} \tilde f)(\eta)\!=\!\intl_{\xi
\subset \eta} \!\!\tilde f(\xi)\, d_\eta \xi, \quad (\rdk
\stackrel{*}{\vp})(\xi)\!=\!\intl_{\eta \supset \xi
} \!\!\stackrel{*}{\vp}(\eta)\, d_\xi \eta;\ee
\be\label {i43q} (R_{k,n-m} \tilde \vp)(\t)\!=\!\intl_{\z
\subset \t} \!\! \tilde \vp(\z)\, d_\t \z, \quad  (\stackrel{*}{R}_{k,n-m}\stackrel{*}{f})(\z)\!=\!\intl_{\t \supset \z
} \!\! \stackrel{*}{f}(\t)\, d_\z \t.\ee
Here $\,d_\eta \xi, \;d_\xi \eta, \;d_\t \z, \;d_\z \t$ denote the relevant probability measures. Then (\ref{876a}) and (\ref{876ab}) imply
\be (F_{m, k}f)(u)\!=\!(R_{m,n-k} \tilde f)(\eta)=(\stackrel{*}{R}_{k,n-m}\stackrel{*}{f})(\z),\ee
\be (\fd \vp)(v)= (R_{k,n-m} \tilde \vp)(\t)=(\rdk
\stackrel{*}{\vp})(\xi).\ee
These equalities hold under the standard assumption
$$ \dim G_{n,m}\le \dim G_{n,n-k}\quad \mbox{\rm or}\quad \dim G_{n,k}\le \dim G_{n,n-m},$$ according to which it is usually assumed $1\le m\le k\le n-m$.

\section{Cosine  and Sine Transforms. Composition formulas}

 Lemma \ref{pmz} suggests
to introduce  the following integral operators: \bea
\label{0mby}(\C^{\a}_{m, k} f)(u)&=&\int_{\vnm} \!\!\!f(v)\,
|v'uu'v|^{(\a-k)/2} \, d_*v,\\ \label{0mbyd}(\cd0
\vp)(v)&=&\int_{\vnk} \!\!\vp(u)\, |v'uu'v|^{(\a-k)/2} \, d_*u,\\
\label{c0mby}(\S^{\a}_{m, k} f)(u)&=&\int_{\vnm} \!\!\!f(v)\, |I_m
-v'uu'v|^{(\a+k-n)/2} \, d_*v,\\
 \label{c0mbyd}(\sd0 \vp)(v)&=&\int_{\vnk} \!\!\vp(u)\,|I_m
-v'uu'v|^{(\a+k-n)/2}  \, d_*u.\eea \[ u \!\in\! V_{n,k}, \qquad v
\!\in\! \vnm, \qquad 1\le m, k\le n-1.\] We call $\C^{\a}_{m, k} f$ and $\S^{\a}_{m, k} f$ the {\it
cosine transform} and the {\it sine transform} of $f$, respectively.
 Integrals $\cd0 \vp$ and $\sd0 \vp$ are called the {\it dual cosine
transform} and the {\it dual sine transform}. The terminology stems from
the fact that, in the case $k=m=1$, when $u$ and $v$ are unit vectors, $$|v'uu'v|=(u\cdot v)^2=\cos^2 \om, \quad |I_m-v'uu'v|=1-(u\cdot v)^2=\sin^2 \om,$$
 where $\om$ is the angle between $u$ and $v$; see also \cite {A, AB, GR, OR3, OR4, Zh2}, regarding  higher-rank analogues of the cosine transform in the language of Grassmannians.

\begin{remark}\label {lp00} {\rm When dealing with operators (\ref{0mby}) and (\ref{0mbyd}), we restrict our consideration to the case $m\le k$, because, if $m>k$, then $|v'uu'v|=0$ for all $v\in \vnm$ and all $u\in\vnk$. Similarly,  for (\ref{c0mby}) and (\ref{c0mbyd}), we assume  $m\le n-k$, because, if $m>n-k$, then
\[|I_m -v'uu'v|=|I_m
-v'\Pr_{\{u\}}v|=|v' \Pr_{\{u\}^\perp } v|=|v' \tilde u \tilde u' v|=0
\]
(here $\tilde u $ is an arbitrary $(n-k)$-frame  orthogonal to $ \{u\}$). The case $k=n$, when $ v'uu'v\equiv I_m$, is also not interesting.
Clearly, \be\label{robp}
(\S^{\a}_{m, k} f)(u)\!=\!(\C^{\a}_{m, n-k} f)(\tilde u)\!=\!\int_{\vnm} \!\!\!f(v)\, |v'\tilde u\tilde
 u'v|^{(\a-(n-k))/2} \, d_*v.\ee
 }\end{remark}

The case $k=m$, when $\C^{\a}_{m, k}$ and $\S^{\a}_{m, k}$  coincide with
their duals, are of particular importance. In this case we denote \bea\qquad
\label{nkmt}(M^{\a} f)(u)&=&\int_{\vnm} \!\!\!f(v)\,
|v'uu'v|^{(\a-m)/2} \, d_*v,\quad 1\!\le\! m\! \le\! n\!-\!1,\\
 \label{0mbyr}(Q^\a f)(u)&=&\int_{\vnm}\!\!\!f(v)\,|I_m\!-\!v'
uu' v|^{(\a+m-n)/2}\, d_*v,\quad 2m\!\le \!n,\eea where $u \in V_{n,m}$.  As we shall see below, $Q^\a$  serves (after
suitable normalization) as a substitute for the Riesz potential
operator in the framework of the corresponding Radon theory.

  The following statement gives precise information about convergence of integrals (\ref{0mby})-(\ref{0mbyr}).
   \begin{theorem}\label {exi}  Let  $f \!\in \!L^1(\vnm)$, $\vp \!\in\!
L^1(V_{n,k})$, $1\le m, k  \le n-1$.

\noindent {\rm (i)} Integrals (\ref{0mby})-(\ref{c0mbyd}) converge absolutely almost everywhere
 if and only if $Re\,\a>m-1$.

 \noindent {\rm (ii)} If $1\le m\le k\le n-1$, then
 \bea \label {lop1}\int_{V_{n,k}}(\C^{\a}_{m, k}f)(u)\,
d_*u&=&c_1\,\int_{V_{n,m}}
f(v)\, d_*v,\\
 \label {lop2}\int_{V_{n,m}}(\cd0 \vp)(v)\, d_*v&=&c_1\,\int_{V_{n,k}} \vp(u)\,
d_*u,\eea
\[c_1=\frac{\Gam_{m} (n/2)\, \Gam_{m} (\a/2)}
{\Gam_{m} (k/2)\,
\Gam_{m}((\a+n-k)/2)}.\]

\noindent {\rm (iii)} If $1\le m\le n-k$, then
\bea \label {lop3}\int_{V_{n,k}}(\S^{\a}_{m, k} f)(u)\,
d_*u&=&c_2\,\int_{V_{n,m}} f(v)\, d_*v, \\
 \label {lop4}\int_{V_{n,m}}(\sd0 \vp)(v)\, d_*v&=&c_2\,\int_{V_{n,k}} \vp(u)\,
d_*u,\eea  \[
c_2=\frac{\Gam_{m} (n/2)\, \Gam_{m} (\a/2)} {\Gam_{m} ((n-k)/2)\,
\Gam_{m}((\a+k)/2)}.\]
 \end{theorem}
 \begin{proof} Equalities (\ref{lop1}) and (\ref{lop2}) hold by Fubini's theorem, owing to Lemma \ref{hbi1} and (\ref{mnv}).   The proof of (\ref{lop3}) and (\ref{lop4})
is similar. It suffices to  note that
\[\int_{\vnk}
 |I_m-v'uu'v|^{(\a+k-n)/2} \, d_*u=\int_{\vnm}
 |v'\tilde u\tilde u'v|^{(\a+k-n)/2} \, d_*v\]
 for any $\tilde u \in V_{n,n-k}$, so that (\ref{mnv}) is  applicable. The validity of (i) follows from the proof of (\ref{lop1})-(\ref{lop4}).
\end{proof}

Functions $(\C^{\a}_{m, k} f)(u)$ and $(\S^{\a}_{m, k} f)(u)$ are right
$O(k)$-invariant. Similarly, $(\cd0 \vp)(v)$ and $(\sd0 \vp)(v)$ are
right $O(m)$-invariant. Hence, transformations
(\ref{0mby})-(\ref{c0mbyd}) actually take functions on the Stiefel
manifolds to functions on the corresponding Grassmann manifolds. If $f$
and $\vp$ are right-invariant in the suitable sense, our operators actually act from one Grassmannian to another.

Below we derive a series of formulas connecting
  cosine, sine,  and Funk transforms. Similar formulas  and their applications in the  case $m=1$ can be found in \cite{Ru4, Ru7}.

\begin{theorem}\label {xcty} Let $f\in L^1(\vnm), \; 1\le m\le k\le n-m$. \\If $Re \, \a\!> \!m\!-\!1$,
then \be \label {gty} \cd0 F_{m,k}f=\fd \C^\a_{m,k} f=c_\a\, Q^{\a+n-k-m} f,\ee
\[
c_\a=\frac{\Gam_m((n-m)/2)\, \Gam_m(\a/2)}{\Gam_m(k/2)\, \Gam_m((\a+n-m-k)/2)}.\]
 If $Re \, \a> k-1$, then
 \be \label {gty7}\cd0 F_{m,k}f=\fd \C^\a_{m,k} f=\tilde c_\a\, M^{\a+m-k}
F_{m}f,\ee
\[ \tilde c_\a=\frac{\Gam_{m} (m/2)\, \Gam_{m} (\a/2)}{\Gam_{m}
(k/2)\, \Gam_{m} ((\a+m-k)/2)}.\]
\end{theorem}
\begin{proof} The equality $\cd0 F_{m,k}f=c_\a\, Q^{\a+n-k-m} f$ in
(\ref{gty}) mimics Lemma \ref{pmz}. The equality $\cd0 F_{m,k}f=\fd \C^\a_{m,k} f$ holds by duality:
 \bea (\fd \C^\a_{m,k} f, \om)&=&(f,\cd0 F_{m,k}\om)=c_\a\, (f, Q^{\a+n-k-m}\om)\nonumber\\
 &=&c_\a\, (Q^{\a+n-k-m}f,  \om), \qquad \om \in C^\infty (\vnm).\nonumber\eea
 Note that by Theorem \ref{exi}, $Q^{\a+n-k-m} f\in L^1 (\vnm)$, because the conditions $ k\le n-m$ and $Re \, \a\!> \!m\!-\!1$ imply  $Re \, \a\!+\!n\!-\!k\!-\!m\!> \!m\!-\!1$.

 To obtain  (\ref{gty7}),
we first write (\ref{gty}) with $k=m$.  This gives
 \be\label {gty9} M^\a
F_{m}f\!=\!\frac{\Gam_m((n-m)/2)\, \Gam_m(\a/2)}{\Gam_m(m/2)\,
\Gam_m((\a+n-2m)/2)}\, Q^{\a+n-2m}f.\ee Then we replace $\a$ by
$\a-k+m$ in (\ref{gty9}) to get $Q^{\a+n-k-m} f$ on the right-hand
side and compare the result with (\ref{gty}).
\end{proof}

\begin{theorem} Let  $f\in L^1(\vnm),\;
 Re \, \a>m-1$.
If $k\le n-m$, $u\in V_{n,k}$, then
\be \label {782}(F_{m, k}M^\a f)(u)= c_{n,k,m}(\a)\,
(\S_{m, k}^{\a+n-k-m} f)(u),\ee
\[c_{n,k,m}(\a)=\frac{\Gam_{m} ((n-k)/2)\,\Gam_{m} (\a/2)}{\Gam_{m} (m/2)\,  \Gam_{m} ((\a+n-k-m)/2)}.\]
If $2m\le n$,  $u\in V_{n,m}$, then
\be \label {782m}(F_{m}M^\a f)(u)\!=\!(M^\a
F_{m}f)(u)\!=\!c_{n,m}(\a)\, (Q^{\a+n-2m} f)(u),\ee
\[c_{n,m}(\a)=\frac{\Gam_{m} ((n-m)/2)\,\Gam_{m} (\a/2)}{\Gam_{m} (m/2)\,  \Gam_{m} ((\a+n-2m)/2)}.\]
\end{theorem}
\begin{proof} By (\ref{876a}), denoting $f_u=f \circ g_u$, we obtain
\bea(F_{m, k}M^\a f)(u)&=&\int_{V_{n-k,m}} (M^\a f_u)\left(\left[\begin{array} {c} \om
\\0
\end{array} \right]\right)\,d_*\om\nonumber\\
&=&\int_{V_{n-k,m}}d_*\om \int_{ V_{n, m}}f_u(v)\,\left | \left[\begin{array} {c} \om
\\0
\end{array} \right]' v \right |^{\a -m}\,d_*v.\nonumber\eea
Set $v=\left[\begin{array} {c} a
\\b
\end{array} \right]$, $a \in \frM_{n-k,m}$, $b\in \frM_{k,m}$ and write $a$ in polar coordinates $$a=zs^{1/2},\quad
 z\in V_{n-k,m}, \quad s=a'a=v'\sig_0\sig'_0 v, \quad \sig_0=\left[\begin{array} {c}
 I_{n-k}
\\0
\end{array} \right]\in V_{n,n-k}.$$
 Changing the order of integration, we have
\[ (F_{m, k}M^\a f)(u)=\int_{ V_{n, m}}f_u(v)\, d_*v\int_{V_{n-k,m}}|\om'a|^{\a -m}\,d_*\om.\] The inner integral equals $c\, |s|^{(\a -m)/2}$, where
\[c=\int_{V_{n-k,m}}|\om'z|^{\a -m}\,d_*\om=\frac{\Gam_{m} ((n-k)/2)\,\Gam_{m} (\a/2)}{\Gam_{m} (m/2)\, \Gam_{m} ((\a+n-k-m)/2)};\]
see (\ref{mnv}). Hence,
\bea (F_{m, k}M^\a f)(u)&=&c\,\int_{ V_{n, m}}\!f_u(v)\, |v'\sig_0\sig'_0 v|^{(\a -m)/2}\,d_*v\nonumber \\
&=&c\,\int_{ V_{n, m}}\!\!f(v)\, |I_m -v'uu' v|^{(\a
-m)/2}\,d_*v,\nonumber \eea as desired.
Equality (\ref{782m}) follows from (\ref{gty}) and
(\ref{782}).
\end{proof}

The  following  useful factorizations hold.

\begin{corollary} \label{85b} Let $f \in L^1 (\vnm)$.

\noindent {\rm (i)} If $ 2m\le n, \; Re \, \a> n-m-1$, then
\be \label {gty4}Q^{\a}f=d_\a\, M^{\a+2m-n} F_{m}f=d_\a\,
 F_{m} M^{\a+2m-n} f,\ee
\[ d_\a=\frac{\Gam_{m} (m/2)\, \Gam_{m} (\a/2)}{\Gam_{m}
((n-m)/2)\, \Gam_{m} ((\a+2m-n)/2)}.\]

\noindent {\rm (ii)} If $ m \le k, \; Re\, \a>k-1$, and $u\in \vnk$, then for any $\tilde u\in \{u\}^\perp$,
\be \label {782a}(\C_{m, k}^\a f)(u)=\tilde d_\a\, (F_{m, n-k}M^{\a+m-k} f)(\tilde u),\ee
\[\tilde d_\a=\frac{\Gam_{m} (m/2)\,\Gam_{m} (\a/2)}{\Gam_{m} (k/2) \,\Gam_{m} ((\a+m-k)/2)}.\]
\end{corollary}
\begin{proof}  {\rm (i)} can be obtained from  (\ref{782m}) if we
replace $\a$ by $\a+2m-n$.  To prove {\rm (ii)}, we re-write (\ref{782}) using (\ref{robp}). Then we replace $k$ by $n-k$ and $\tilde u$ by $u$.
 This gives
\[(\C_{m, k}^{\a+k-m} f)(u)=\check d_\a\, (F_{m, n-k}M^{\a} f)(\tilde u), \qquad \tilde u\in \{u\}^\perp,\]
\[\check d_\a=\frac{\Gam_{m} (m/2)\,\Gam_{m} ((\a+k-m)/2)}{\Gam_{m} (k/2) \,\Gam_{m} (\a/2)}.\]
Now we replace $\a+k-m$ by $\a$, and we are done.
\end{proof}

Formula (\ref{782a}) was obtained by the author several years ago and reported to S. Alesker, who gave another proof of it; cf.  \cite [Proposition 1.2]{A}. A similar factorization in the language of Grassmannians is presented in \cite[Lemma 3.4]{Zh2}, however, without proof and without explicit constant.

The following statement
extends Theorem \ref{xcty} to $\a=0$; cf. \cite[Theorem 2.4]{Gri}, where the close result was obtained in the language of Grassmannians in the spectral form.
\begin{theorem} \label {plz} Let $1\le k,m\le n-1;
\; 2m \le n-k$. Then \be \label {arn}  \fd F_{m, k}f\!=\!
\tilde c\,Q^{n-k-m} f, \quad \tilde c\!=\!\frac{2^m
\pi^{(n-m)m/2}\,\Gam_{m}((n\!-\!k)/2)}{\Gam_{m}(n/2)\, \Gam_{m}
((n-k-m)/2)}.\ee
\end{theorem}
\begin{proof}
Let $g_v$ be an orthogonal transformation which sends $v_0=
\left[\begin{array} {c} 0  \\  I_m
\end{array} \right]$ to $v\in \vnm$. We denote $f_v (w)= f(g_v w)$.
By (\ref{la3v}) and (\ref{la3vd}), \bea (\fd F_{m, k}f)(v)
&=&\int_{\check {K}_0 }(F_{m, k}f)(g_v \rho \check u_0)\, d\rho
\nonumber\\
&=&\int_{\check {K}_0 }d\rho\int_{O(n-k)} f_v\left(\rho
\z'\left[\begin{array} {cc} \gam & 0
\\ 0 & I_{k}
\end{array} \right]\, \left[\begin{array} {c}  I_m \\  0 \end{array}
\right] \right )\, \, d\gam,\nonumber\eea where $\z\in O(n)$ is
defined by (\ref{78v}) and satisfies $\z' u_0 =\check{u}_0$. Hence,
 \[(\fd F_{m,
k}f)(v)=\int_{O(n-m)} d\del\int_{V_{n-k,m}}f_v \left
(\left[\begin{array} {cc} \del & 0
\\ 0 & I_m
\end{array} \right]\z'\left[\begin{array} {c}  w\\  0 \end{array}
\right]\right )\,d_* w.\]
 Using the bi-Stiefel
decomposition (\ref{2.11}) (replace $n$ by $n-k$, and $k$ by $m$),
the last expression can be written as \bea
&&\frac{1}{\sig_{n-k,m}}\int_{O(n-m)}  d\del \int_0^{I_m} d\nu(r) \int_{ V_{m, m}}d\gam\nonumber\\
&&\times \int_{ V_{n-k-m, m}} f_v \left (\left[\begin{array} {cc}
\del & 0
\\ 0 & I_m
\end{array} \right]\z'
 \left[\begin{array} {c}  \gam r^{1/2}\\ u(I_m -r)^{1/2} \\ 0 \end{array}
\right]\right )\, du. \nonumber\\
&&=\frac{1}{\sig_{n-k,m}}\int_0^{I_m} d\nu(r) \int_{ V_{m, m}}d\gam
\int_{ V_{n-k-m, m}} du\nonumber\\
&&\times \int_{O(n-m)}   f_v \left ( \left[\begin{array} {c} \del
\left[\begin{array} {c}  0 \\   u\end{array} \right] (I_m
-r)^{1/2}\\
\gam r^{1/2}
\end{array}
\right]\right )\, d\del \nonumber\\
&&=\frac{\sig_{n-k-m,m}}{\sig_{n-k,m}}
\int_0^{I_m} d\nu(r) \int_{ V_{m, m}}d\gam\nonumber\\
&&\times \int_{ V_{n-m, m}}  f_v \left ( \left[\begin{array} {c}
  \theta (I_m -r)^{1/2}\\
\gam r^{1/2} \end{array} \right]\right )\, d\theta. \nonumber\eea
 Here
\[ d\nu(r)\!=\!2^{-m} |r|^{m/2-d} |I_m \!-\!r|^{(n-k-m)/2-d}\, dr, \qquad 2m\!\le \!n\!-\!k.\]  Now we change variables $r \to I_m -r$ and use
(\ref{2.11}) (with $k$ replaced by $n-m$) in the opposite
 direction. We obtain
 \bea (\fd F_{m,
k}f)(v)&=&\frac{\sig_{n-k-m,m}}{\sig_{n-k,m}} \int_{ V_{n, m}}  f_v
(w) |w'(I_n-v_0 v'_0) w|^{-k/2}\, dw\nonumber\\
&=&\frac{\sig_{n-k-m,m}\, \sig_{n,m}}{\sig_{n-k,m}} \int_{ V_{n, m}}
f (w) |I_m-w'vv'w|^{-k/2}\, d_*w.\nonumber\eea Owing to
(\ref{0mbyr}) and (\ref{2.16}), this is exactly what we need.
\end{proof}

 \begin{remark} {\rm The  assumption $2m \le n-k$ in Theorem \ref{plz} is necessary
for absolute convergence of the integral on the right-hand side in (\ref{arn}), whereas the  left hand-side is finite a.e. under the
weaker assumption $m \le n-k$, which is sharp. By Lemma  \ref{mkloi} below, the
condition $2m \le n-k$ can be eliminated if we interpret (\ref{arn}) in the $\S'$-sense.}\end{remark}

\section { Cosine  Transforms via the Fourier Analysis}\label{mlz5}

 In this section we proceed to develop  the theory of the Funk, cosine and sine transforms
  using the Fourier analysis in the ambient matrix space. The consideration
 essentially relies on the idea of analytic continuation, when the one-dimensional exponent $\a$
 is replaced by a vector-valued complex parameter  $\lv\in\bbc^m$.

 \subsection{The composite cosine  transform}

 Let $\lv=(\lam_1,\ldots, \lam_m)\in\bbc^m$ and let $(\cdot)^{\lv}$ be the corresponding
 composite power function associated with the cone $\Omega$; see Section \ref{secpf}. We consider {\it the
 composite cosine transform}
\be\label{tnf}(T_{k,m}^{\lv} \vp)(v)=\int_{\vnk} \vp(u) \, (v'uu'v)^{\lv} \,
d_*u, \qquad v\in\vnm,\ee
which was studied in \cite{OR3, OR4} when $k=m$. The dual  cosine transform (\ref {0mbyd}) is a particular case of (\ref{tnf}), corresponding to $\lam_1\!=\!\ldots\!=\!\lam_m\!=\a\! -\!k$.
\begin{lemma}\label {rexi}  Let $1\le m\le k\le n-1$,  $\vp \in L^1(\vnk)$,
\be \label {frle} \frL=\{\lv\in\bbc^m :
Re\,\lam_j>j-k-1 \quad \text{\it for each} \quad j=1,\dots, m \}.\ee
 The integral $(T_{k,m}^{\lv} \vp)(v)$
converges absolutely for almost all   $v\in\vnm$ if and only if $\lv \in
\frL$ and represents an analytic function of  $\lv$ in this
domain.  Moreover,
\be \int_{\vnm}(T_{k,m}^{\lv} \vp)(v)\, d_*v= \frac{\Gam_m (m/2)\,\Gam_\Om(\lv+\kv)}{\Gam_m (k/2)\, \Gam_\Om(\lv+\nv)}\, \int_{\vnk} \vp(u)\, d_*u.\ee
\end{lemma}

This statement follows immediately from Lemma \ref{exl} in Appendix.

\subsection{The  complementary Radon   transform}
This is a new transformation that does not occur in \cite {OR4, OR3} in the case $k=m$. It takes a function $\vp (u)$ on $\vnk$ to a function $(A_{k,m} \vp)(v)$ on $\vnm$ by the formula
\bea \label {akm} \qquad(A_{k,m} \vp)(v)&=&\int_{V_{n-m, k-m}}\!\!\!\!\vp \left (g_v \left[\begin{array} {cc} a & 0 \\  0 & I_{m} \end{array} \right]\right )\, d_*a \\&=& (F_{k-m,m} \vp ([\cdot, v]))(v), \qquad  1\le m\le k\le n-1,\nonumber\eea
where  $ g_v$ is an orthogonal transformation that sends $v_0\!=\! \left[\begin{array} {c}  0 \\  I_{m} \end{array} \right]$ to $v\!\in \!\vnm$.
This can be regarded as a ``partial Funk transform'' (cf.  (\ref{876a})), where
  $\vp$ is integrated over all orthonormal $k$-frames $u$ in $\bbr^n$, the last $m$ columns of which are replaced by   $v$, and the first $k-m$ columns are orthogonal to $v$. If $k=m$, then $A_{k,m}$ reduces to the identity operator.

\begin{lemma}\label {exix}  Let $1\le m\le k\le n-1, \; \vp\in L^1(\vnk)$. Then
\[ \int_{\vnm} (A_{k,m} \vp)(v)\, d_*v=\int_{\vnk} \vp(u)\, d_*u.\]
 \end{lemma}
\begin{proof}
\bea \l.h.s&=&\int_{O(n)}dg\int_{O(n-m)}\vp \left (g \left[\begin{array} {cc} \del & 0 \\  0 & I_{m} \end{array} \right]
\left[\begin{array} {c}  0 \\  I_{k} \end{array} \right]\right )\, d\del\nonumber\\
&=&\int_{O(n)}\vp \left (g
\left[\begin{array} {c}  0 \\  I_{k} \end{array} \right]\right )dg=r.h.s.\nonumber\eea
\end{proof}

\subsection{The basic functional equation}

The following lemma from  \cite{OR3} will be needed. We present it with proof in view of its fundamental importance.
We denote  \bea\label{abs} \qquad {\bf
\Lam}&=&\{\lv\!\in\!\bbc^m :
Re\,\lam_j\!>\!j\!-\!n\!-\!1 \quad \text{\it for each} \quad j\in \{1,\dots, m \}, \\
\label{pset}
{\bf \Lam}_0&=&\{\lv\!\in\!\bbc^m : \lam_j=j-n-l \quad \text{\it for some}\quad j \in \{1,\dots, m\}\\
 &{}& \text{\it and some} \quad l \in \{1,3,5,\dots \} \}.
 \nonumber
 \eea

\begin{lemma}\label{l3.1}  Let $f\in L^1(\vnm)$, $\phi\in \S(\Ma)$.
The integrals
\be \label {mlppi}(r^\lv f, \phi)\!=\!\intl_{\Ma}\!\!
r^{\lv} \, f(v) \, \overline{\phi(x)} \,  dx,\quad (r_*^\lv f, \phi)\!=\!\intl_{\Ma}\!\! r_*^{\lv} \, f(v) \,
\overline{\phi(x)}  \, dx\ee
 are  absolutely convergent if and only if $\lv \in
{\bf\Lam}$ and extend
 as  meromorphic functions of $ \lv$
with the  polar set ${\bf \Lam}_0 $.  The normalized
integrals
\be\label{nzi}\frac{
(r^\lv f, \phi)}{\Gam_\Om(\lv+\nv)}, \qquad
\frac{(r_*^\lv f, \phi)}{\Gam_\Om(\lv+ \nv)},
\ee $\nv=(n, \ldots, n)$, extend as   entire functions of $\lv$.
 \end{lemma}
 \begin{proof}  Consider the first integral.
 We set
 $ x=vt$, $ v \in
\vnm$,   $ t \in T_m$, and make use of Lemma \ref{sph}. By taking
into account that $x'x=t't$ and $t(t't)^{-1/2} \in O(m)$, owing to
(\ref{mku5}), we obtain \be\label{dis1} (r^\lv f, \phi)=
\intl_{\bbr^m_+} F(t_{1,1},\dots , t_{m,m}) \prod\limits_{j=1}^m
t_{j,j}^{\lam_j+n-j} dt_{j,j}\, , \ee where
\[
F(t_{1,1},\dots , t_{m,m}) = \intl_{\bbr^{m(m-1)/2}}
dt_*\intl_{\vnm} f(v)\overline{\phi(v t)}\,d v, \quad
dt_*=\prod\limits_{i<j} dt_{i,j}.\] Since $F$ extends as an even
Schwartz function in each argument, it can be written as
$
F(t_{1,1},\dots , t_{m,m}) =F_0(t^2_{1,1},\dots , t^2_{m,m})$,
where $F_0\in\S(\bbr^m)$ (use, e.g., Lemma 5.4 from \cite[p.
56]{Tr}). Replacing $t_{j,j}^2$ by $s_{j,j}$, we represent
(\ref{dis1}) as a direct product of one-dimensional distributions
\be\label{z-h}(r^\lv f, \phi)=(\prod\limits_{j=1}^m
(s_{j,j})_+^{(\lam_j+n-j-1)/2},\;  F_0(s_{1,1},\dots , s_{m,m})).\ee
It follows that  the the first integral in  (\ref{mlppi}) is absolutely convergent
provided $Re\,\lam_j>j-n-1$, i.e., $\lv \in {\bf\Lam}$. The
condition  $\lv \in {\bf\Lam}$ is strict. To see this, we choose
$f\equiv 1$ and $\phi(x)=e^{-{\rm tr} (x'x)}$ so that \be
\label{ex}(r^\lv f, \phi)=\intl_{\Ma} (x'x)^{\lv} \, e^{-{\rm tr}
(x'x)}\,dx=2^{-m}\sig_{n,m}\Gam_\Om(\lv+\nv).\ee Furthermore, since
$(s_{j,j})_+^{(\lam_j+n-j-1)/2}$ extends as a meromorphic
distribution with the only poles  $\lam_j= j-n-1, j-n-3, \dots \;$,
then, by the fundamental Hartogs theorem \cite {Sha}, the function $
\lv \to (r^\lv f, \phi)$ extends  meromorphically
with the  polar set ${\bf \Lam}_0$. By the same reason, a direct
product of the normalized distributions $$
(s_{j,j})_+^{(\lam_j+n-j-1)/2}/\Gam ((\lam_j+n-j+1)/2)$$ is an
entire function of $\lv$.

Let us consider the second integral in   (\ref{mlppi}). Changing variable
$x=y\om$, where $\om$ is a matrix from (\ref{ommp}),
we obtain $
r_*=(x'x)_*=(\om y'y\om)_*=y'y$. Hence (set $y=u\t, \; u \in \vnm,
\; \t \in T_m$),
\bea (r_*^\lv f, \phi)&=&\intl_{\Ma} (y'y)^{\lv}
\, f(y\om (\om
y' y \om)^{-1/2}) \, \overline{\phi(y\om)} \, dy \nonumber \\
&=& \intl_{\bbr^m_+} \Phi(\t_{1,1},\ldots , \t_{m,m})
\prod\limits_{j=1}^m \t_{j,j}^{\lam_{j}+n-j} d\t_{j,j}\,
\nonumber\eea where, as above (note that $\t\om \, (\om \t'\t
\om)^{-1/2} \in O(m)$), \bea \Phi(\t_{1,1},\ldots , \t_{m,m}) &=&
\intl_{\bbr^{m(m-1)/2}} d\t_*\intl_{\vnm} f(u) \, \overline{\phi(u
\t\om)}\,d u\nonumber \\
&=&\Phi_0(\t^2_{1,1},\dots , \t^2_{m,m}), \qquad
\Phi_0\in\S(\bbr^m). \nonumber \eea This gives
\[(r_*^\lv f, \phi)=(\prod\limits_{j=1}^m
(s_{j,j})_+^{(\lam_{j}+n-j-1)/2},\; \Phi_0(s_{1,1},\dots ,
s_{m,m})),\] and the result follows as in the previous case.
\end{proof}

 We introduce the following normalized extensions of $(A_{k,m} \vp)(v)$ and
 $(T_{k,m}^{\lv} \vp)(v)$ from  $\vnm$ to the ambient matrix space $\Ma$:
\bea\label{fil} (\tilde A_{k,m}^{\lv}  \vp)(x)&=&  \frac{ r_*^{-\lv_*-\nv}}{
\Gam_\Om(-\lv_*)} \, (A_{k,m} \vp)(v),\\(\tilde T_{k,m}^{\lv} \vp)(x)&=&\frac{ r^{\lv}}{
\Gam_\Om(\lv+\kv)} \, (T_{k,m}^{\lv} \vp)(v),\eea
where $x\in \Ma, \; x=vr^{1/2}, \; r=x'x\in \Omega, \; v \in \vnm$.

The next theorem plays a key role in the whole paper.
\begin{theorem}\label{gen} Let $\vp$ be an integrable right $O(m)$-invariant function on $\vnm$, $\phi\in
 \S(\Ma)$, $1\le m\le k \le n-1$. Then for every $ \lv\in\bbc^m$,
 \be\label{eq5}
(\tilde T_{k,m}^{\lv} \vp, \F \phi) \!=\!\check c \, (\tilde A_{k,m}^{\lv}  \vp, \phi), \quad
\check c \!=\!\frac{2^{nm+|\lv|}\, \pi^{nm/2}\,\Gam_m (n/2)}{\Gam_m (k/2)}\,,\ee
where both sides  are understood in the sense of analytic continuation.
\end{theorem}
\begin{proof}  It suffices to show the following.

(a) Equality (\ref{eq5}) holds for all $ \lv$ in a certain domain $\frD \subset \bbc^m$,
where both sides of (\ref{eq5}) are analytic functions of $ \lv$;

(b) The right-hand side of (\ref{eq5}) extends from $\frD$ to all $ \lv\in\bbc^m$ as an entire function.

Let $\frD=\frL \cap \tilde\frL$, where
\bea  \frL&=&\{\lv\in\bbc^m :
Re\,\lam_j>j-k-1 \quad \text{\it for each} \quad j=1,\dots, m \},\nonumber\\
\tilde\frL&=&\{\lv\in\bbc^m :
Re\,\lam_j<j-m \quad \text{\it for each} \quad j=1,\dots, m \}\nonumber\eea
(note that  $\frD$ is nonempty !). Absolute convergence of  the integral
\[ (\tilde A_{k,m}^{\lv}  \vp,\, \phi)=\int_{\Ma} (\tilde A_{k,m}^{\lv}  \vp)(x) \overline{\phi (x)}\, dx\]
 for all $\lv \in \tilde\frL$ and extendability of this expression by analyticity to all $ \lv\in\bbc^m$
 follow from  Lemma \ref{l3.1}, owing to Lemma \ref {exix}: just change the notation in the second integral in (\ref{mlppi})  and make use of the equality $ (\lv_\ast)_j=\lam_{m-j+1}$ (the latter gives $Re\,\lam_j<j-m$). Absolute convergence and analyticity of  the left-hand side of (\ref{eq5}) for all $\lv \in \frL$ follow from Lemma \ref{l3.1} too, thanks to Lemma \ref {exi} (note that $\frL \subset  {\bf
\Lam}$, cf. (\ref{abs})).  Thus, it remains to prove (\ref{eq5})  for  $ \lv\in \frD$.

Assuming $\lv\in \frL$, by (\ref{pr6}) we have
\be \label {klop}I\equiv (\tilde T_{k,m}^{\lv} \vp, \F \phi)=\intl_{\vnk} \vp(u) J(u)\, d_*u,\ee
\[ J(u)=\frac{1}{\Gam_{\Omega} (\lv\!+\!\kv)}\intl_{\Ma}  (x'uu'x)^{\lv} \,\overline{(\F\phi) (x)}\,dx, \qquad u \in \vnk.\]
Let $ g_u$ be an orthogonal transformation that sends $u_0\!=\! \left[\begin{array} {c}  0 \\  I_{k} \end{array} \right]$ to $u\!\in \!\vnk$. We set $x=g_u\left[\begin{array} {c}  a \\  b \end{array} \right]$, $a\in \frM_{n-k,m}$, $b\in  \frM_{k,m}$. Then, by (\ref{4.9}),
\bea  J(u)&=&\frac{1}{\Gam_{\Omega} (\lv\!+\!\kv)}\intl_{\frM_{k,m}} (b'b)^{\lv} db\intl_{\frM_{n-k,m}}\overline{(\F\phi) \left(g_u \left[\begin{array} {c}  a \\  b \end{array} \right]\right )}\,da\nonumber\\
&=&\frac{1}{\Gam_{\Omega} (\lv\!+\!\kv)}\intl_{\frM_{k,m}} (b'b)^{\lv} \,\overline{(\R_k \F\phi) (u,b)}\, db.\nonumber\eea
Owing to (\ref{eq8})  and Remark \ref{bnuk} (we recall that $\lv  \in \frD$ !), the last expression can be written as
\bea  &&\frac{2^{km+|\lv|}\pi^{km/2}}{\Gam_{\Omega} (-\lv_*)}
\intl_{\frM_{k,m}} (b'b)_*^{-\lv_* -\kv} \,\overline{(\tilde \F^{-1}(\R_k \F\phi)(u, \cdot)) (b)}\, db\nonumber\\
&&=\frac{2^{|\lv|}\pi^{-km/2}}{\Gam_{\Omega} (-\lv_*)}\,\intl_{\frM_{k,m}} (b'b)_*^{-\lv_* -\kv} \,
\overline{(\tilde \F(\R_k \F\phi)(u, \cdot)) (b)}\, db.\nonumber\eea
By the Projection-Slice Theorem (see (\ref{4.20})),
\be\label {lpoi} (\tilde \F(\R_k \F\phi)(u, \cdot)) (b)=(\F\F\phi)(ub)=(2\pi)^{nm}\phi(-ub).\ee
Hence,
\[ J(u)\!=\!c_\lam \!\intl_{\frM_{k,m}}\!\!\!(b'b)_*^{-\lv_* -\kv} \,\overline{\phi(ub)}\, db, \quad c_\lam\!=\! \frac{2^{nm+|\lv|}\pi^{m(n-k/2)}}{\Gam_{\Omega} (-\lv_*)}.\]
This calculation enables us to transform (\ref{klop}) as follows:
\bea I&=&c_\lam \intl_{\vnk} \vp(u)  d_*u\intl_{\frM_{k,m}}\!\!\!(b'b)_*^{-\lv_* -\kv} \,\overline{\phi(ub)}\, db\nonumber\\
&=&2^{-m} c_\lam \intl_{\Omega} r_*^{-\lv_* -\kv} |r|^{(k-m-1)/2} \psi (r)\, dr, \nonumber\eea
where
\bea
\psi (r)&=&\intl_{V_{k,m}} d\eta\intl_{\vnk} \vp(u) \,\overline{\phi(u\eta r^{1/2})}\,  d_*u \quad \mbox {\rm (set $\eta_0=\left[\begin{array} {c}  0 \\  I_{m} \end{array} \right]\in V_{k,m}$)}\nonumber\\
&=&\sig_{k,m}\intl_{O(k)} d\gam\intl_{\vnk} \vp(u) \overline{\phi(u\gam\eta_0 r^{1/2})}\,  d_*u\nonumber\\
&=&\sig_{k,m}\intl_{\vnk} \vp(u) \,\overline{\phi(u\eta_0 r^{1/2})}\,  d_*u.\nonumber\eea
Using notation from (\ref{klop7}) and (\ref{klop8}), we continue
\bea
\psi (r)&=&\sig_{k,m}\intl_{O(n)} \vp(g u_0) \,\overline{\phi(g u_0\eta_0 r^{1/2})}\,dg\nonumber\\
&{}& \mbox{\rm (note that $u_0\eta_0=v_0=\left[\begin{array} {c}  0 \\  I_{m} \end{array} \right]\in V_{n,m}$ and $\check K_0 v_0=v_0$)}\nonumber\\
&=&\sig_{k,m}\intl_{O(n)} \vp(g u_0) \, dg\intl_{\check K_0} \overline{\phi(g \rho'v_0 r^{1/2})}\,  d\rho\nonumber\\
&=&\sig_{k,m}\intl_{\check K_0} d\rho\intl_{O(n)} \vp(g u_0)  \,\overline{\phi(g \rho'v_0 r^{1/2})}\,  dg\nonumber\\
&=&\sig_{k,m} \intl_{O(n)}\overline{\phi(\lam v_0 r^{1/2})}\,  d\lam
\intl_{\check K_0} \vp(\lam \rho  u_0) \, d\rho\nonumber\\
&=&\frac{\sig_{k,m}}{\sig_{n,m}} \,\intl_{\vnm} \overline{\phi(v r^{1/2})}\,  dv
\intl_{O(n-m)}\vp  \left (g_v \left[\begin{array} {cc} \del & 0 \\  0 & I_{m} \end{array} \right]
\left[\begin{array} {c}  0 \\  I_{k} \end{array} \right]\right )\, d\del\nonumber\\
&=&\label {45x} \frac{\sig_{k,m}}{\sig_{n,m}} \,\intl_{\vnm} \overline{\phi(v r^{1/2})}\, (A_{k,m}\vp)(v)\, dv.\eea

Hence (use (\ref{pr1}) and note that $|r|=|r_*|$),
\bea I&=&\frac{ c_\lam \sig_{k,m}}{2^{m}\sig_{n,m}} \intl_{\Omega} r_*^{-\lv_*} |r|^{-(m+1)/2}\, dr\intl_{\vnm} \overline{\phi(v r^{1/2})}\, (A_{k,m}\vp)(v)\, dv\nonumber\\
&=&\check c\,\intl_{\Ma} (\tilde A_{k,m}^{\lv}  \vp)(x)\, \overline{\phi(x)}\,dx, \nonumber\eea
as desired.
\end{proof}

The following corollaries hold in the  case
$\lam_1=\ldots=\lam_m=\a -k$. We denote
\be\label{132}
 (E_\lam f)(x)=|r|^{\lam/2} f(v), \qquad x=vr^{1/2}\in \Ma.\ee
\begin{corollary} \label {nkil} Let $\vp$ be an integrable right $O(k)$-invariant function on $\vnk$, $\phi\in
 \S(\Ma)$, $1\le m\le k \le n-1$. Then for every $ \a\in\bbc$,
 \be\label{eq5y}
 \left(\frac{E_{\a-k} \cd0 \vp}{\Gam_m (\a/2)}, \F \phi\right)=\check c_1 \, \left (\frac{E_{k-\a-n}  A_{k,m} \vp}{\Gam_m ((k-\a)/2)},  \phi\right),\ee
\[ \check c_1 =\frac{2^{m(n+\a-k)}\, \pi^{nm/2}\,\Gam_m (n/2)}{\Gam_m (k/2)},\]
where both sides  are understood in the sense of analytic continuation.
\end{corollary}

In particular, if $k=m$, the following statement holds for the cosine transform
\[(M^{\a} f)(u)=\int_{\vnm} \!\!\!f(v)\,
|v'uu'v|^{(\a-m)/2} \, d_*v.\]
\begin{corollary} Let $f$ be an integrable right $O(m)$-invariant function on $\vnm$, $\phi\in
 \S(\Ma)$, $1\!\le\! m\! \le\! n\!-\!1$. Then for every $ \a\in\bbc$,
 \be\label{eq5ya}
 \left(\frac{E_{\a-m} M^{\a} f}{\Gam_m (\a/2)}, \F \phi\right)=\check c_2 \, \left (\frac{E_{m-\a-n}  f}{\Gam_m ((m-\a)/2)},  \phi\right),\ee
\[ \check c_2 =\frac{2^{m(n+\a-m)}\, \pi^{nm/2}\,\Gam_m (n/2)}{\Gam_m (m/2)},\]
where both sides   are understood in the sense of analytic continuation.
\end{corollary}

\begin{remark} {\rm For better understanding  of (\ref{eq5y}) (and  also (\ref{eq5ya})) we note that the
 domains, where the left-hand side and
the right-hand side  of this equality exist as absolutely
convergent integrals, have no points in common, when  $m>1$. To  implement analytic continuation we had to  switch from  $\a\in \bbc$ to  $\lv\in \bbc^m$.}
\end{remark}

\begin{corollary} \label{cr1}   Let $\vp$ be an integrable right $O(k)$-invariant function on $\vnk$, $1\le m\le k \le n-1$.
Then for any  $\om \in C^\infty (\vnm)$, the function
$$\z_*(\a)= \Big(\frac{1}{\Gam_m (\a/2)}\, \cd0 \vp, \om\Big)$$ extends  to all $\a\in \bbc$ as an entire function.
\end{corollary}

\begin{proof} Let $\psi (r)$ be a nonnegative $C^\infty$ function with compact support away from the boundary of the cone $\Omega$. Choose $\F\phi$ in
(\ref{eq5y}) so that $(\F\phi)(x)=\psi (r)\om (v)$, $x=vr^{1/2}$, $v\in \vnm$.
Then  for $Re \, \a > m-1$, the left-hand side of (\ref{eq5y}) becomes $\z_*(\a)h_\psi (\a)$,
\[ \qquad h_\psi
 (\a)=2^{-m}\int_{\Omega} |r|^{(n+\a-m-k-1)/2} \overline{\psi (r)}\,dr.\]
This gives \be\label{124} a.c.\, \z_*(\a)=a.c.\, \frac{1}{h_\psi (\a)}\, \left(\frac{E_{\a-k} \cd0 \vp}{\Gam_m (\a/2)}, \psi \om\right).\ee
Owing to uniqueness of analytic continuation and Corollary \ref{nkil}, the right-hand side of
(\ref{124}) is well defined and independent of $\psi$.
\end{proof}

The following statement holds by duality.

\begin{corollary} \label{cr12}   Let $f$ be an infinitely differentiable right $O(m)$-invariant function on $\vnm$, $1\le m\le k \le n-1$.
Then for any  $\om \in C^\infty (\vnm)$, the function
$$\z(\a)= \Big(\frac{1}{\Gam_m (\a/2)}\, \C_{m,k}^\a f, \om\Big)$$ extends  to all $\a\in \bbc$ as an entire function.
\end{corollary}

\subsubsection{Remarks and conjectures}${}$\hfil

{\bf 1.} We believe that Corollary \ref{cr12} holds for every $f\in L^1(\vnm)$. This would follow from our next conjecture.

{\bf 2.} It is natural to expect that for infinitely differentiable  $f$ and $\vp$, the normalized functions
\be \a \to \frac{1}{\Gam_m (\a/2)}\, (\C_{m,k}^\a f)(u),\qquad \a \to\frac{1}{\Gam_m (\a/2)}\,(\cd0 \vp)(v)\ee
extend as entire functions of $\a$ pointwise, that is, for every $u$ and $v$, and, moreover, these extensions are $C^\infty$ functions of $u$ and $v$, respectively. We suppose that the proof of this fact can be given by using decomposition in Stiefel/Grassmann harmonics (see, e.g, \cite {Ge, Str75, Str86, TT} for this theory).

{\bf 3.} Analytic continuation of integrals, like $\C_{m,k}^\a f$, was briefly discussed in \cite[p. 368]{GGR}, where it was suggested to replace such integrals by those over the matrix space $\frM_{m,n-m}$, so that $\dim \frM_{m,n-m}=m(n-m)=\dim G_{n,m}$. However, the corresponding calculations were skipped in that paper. Our impression is that the actual procedure is  more complicated; see  Section \ref  {bhub}
for further discussion.

\subsection{A functional equation for the dual Funk  transform}
The next theorem  specifies Corollary \ref{nkil} for $\a=0$. According to this theorem,  the dual Funk  transform $\fd$ can be regarded  as a member of the  analytic family $\{\cd0\}$ if the latter is suitably normalized.

\begin{theorem} \label {nkild} Let $\vp$ be an integrable right $O(m)$-invariant function on $\vnm$, $\phi\in
 \S(\Ma)$, $1\le m\le k \le n-1$. Then
 \be\label{eq5yq}
(E_{-k}\fd \vp, \F \phi)=\tilde c \, (E_{k-n}  A_{k,m} \vp,  \phi),\ee
where $A_{k,m}$ is the operator (\ref{akm}),
\[ \tilde c =\frac{ 2^{m(n-k)}\, \pi^{nm/2}\, \Gam_m ((n-k)/2)\,}{\Gam_m (k/2)}.\]
\end{theorem}
\begin{proof} Passing to polar coordinates, we obtain
\[(E_{-k}\fd \vp, \F \phi)=2^{-m}\sig_{n,m}\intl_{\vnm}(\fd \vp)(v)\,h (v)\, d_*v,\]
\[h (v)=\intl_{\Omega}\overline{(\F \phi)(vr^{1/2})}\,|r|^{(n-k-m-1)/2} \,dr.\]
Hence, by  duality (\ref{009a}),
\be\label {op8m} (E_{-k}\fd \vp, \F \phi)=2^{-m}\sig_{n,m}\intl_{\vnk} \vp (u) (F_{m,k}h)(u)\, d_*u,\ee
where, by (\ref{876a}),
\bea (F_{m,k}h)(u)&=&\intl_{V_{n-k,m}} \!\!\!\!\!h\left(g_u
\left[\begin{array} {c} \om \\0 \end{array} \right]\right)\,d_*\om\nonumber\\
&=&\intl_{V_{n-k,m}}d_*\om\intl_{\Omega}\overline{(\F \phi)\left(g_u
\left[\begin{array} {c} \om \\0 \end{array} \right]r^{1/2}\right)}
\,|r|^{(n-k-m-1)/2} \,dr\nonumber\\
&=&\frac{2^m}{\sig_{n-k,m}}\intl_{\frM_{n-k,m}}\overline{(\F \phi)\left(g_u
\left[\begin{array} {c} y \\0 \end{array} \right]\right)}\,dy\nonumber\\&=&\frac{2^m}{\sig_{n-k,m}}\, \overline{(\R_k \F \phi)(u,0)}.\nonumber\eea
The last expression  can be regarded as analytic continuation of the Riesz distribution (\ref{rd}), so that
\[(F_{m,k}h)(u)=\frac{2^m}{\sig_{n-k,m}}  \,\underset
{\a=0}{a.c.} \, \frac{1}{\gam_{k,m} (\a)}
  \intl_{\frM_{k,m}} \overline{(\R_k \F \phi)(u,z)}\,|z|_m^{\a-k}  \,dz.\]
 Now (\ref {fou}) yields
\[(F_{m,k}h)(u)=\frac{2^{m-km}\pi^{-km}}{\sig_{n-k,m}}\,  \underset
{\a=0}{a.c.} \,
  \intl_{\frM_{k,m}} \overline{\tilde \F [(R_k \F \phi)(u,\cdot)](y)}\,|y|_m^{-\a}  \,dy.\]
 Hence, by (\ref{lpoi}),
 \[(F_{m,k}h)(u)=\frac{2^{m(1-k+n)}\pi^{m(n-k)}}{\sig_{n-k,m}}\,
  \intl_{\frM_{k,m}} \overline{\phi (uy)}\, dy.\]
Using (\ref{op8m}) and passing to polar coordinates, we obtain
\bea (E_{-k}\fd \vp, \F \phi)&=&\!\frac{(2\pi)^{nm-km}\, \sig_{n,m}}{\sig_{n-k,m}}\,\intl_{\vnk} \!\vp (u)  d_*u\!\!\intl_{\frM_{k,m}} \!\overline{\phi (uy)}\, dy\nonumber\\
&=&\! \frac{(2\pi)^{nm-km} \sig_{n,m}}{2^m\,\sig_{n-k,m}\, }\, \intl_{\Omega} \!|r|^{(k-m-1)/2}\, \psi (r)\, dr, \nonumber\eea
where, by (\ref{45x}),
\[\psi (r)= \frac{\sig_{k,m}}{\sig_{n,m}} \,\intl_{\vnm} \overline{\phi(v r^{1/2})}\, (A_{k,m}\vp)(v)\, dv.\]
Setting $x=vr^{1/2}$, we arrive at (\ref{eq5yq}), as desired.
\end{proof}
\begin{corollary} \label {nkilf} Let $\vp$ be an integrable right $O(k)$-invariant function on $\vnk$, $\phi\in
 \S(\Ma)$, $1\le m \le k\le n-1$. Then
 \be\label{ceq5yf}
(E_{-k} \fd \vp, \phi)=d _0 \,  \underset
{\a=0}{a.c.} \,\left (\frac{E_{\a -k}  \cd0 \vp}{\Gam_{m} (\a/2)},  \phi\right),\ee
\[ \tilde d_0 =\frac{\Gam_m ((n-k)/2)\Gam_{m} (k/2)\,}{\Gam_m (n/2)}.\]
\end{corollary}
\begin{proof} The statement follows immediately from (\ref{eq5yq}) and (\ref{eq5y}).
\end{proof}

In the case $k=m$ we have the following
\begin{corollary} \label {nkilz} Let $f$ be an integrable right $O(m)$-invariant function on $\vnm$, $\phi\in
 \S(\Ma)$, $1\le m \le n-1$. Then
 \be\label{ceq5yq}
(E_{-m} F_m f, \F \phi)=\tilde c_0 \, (E_{m-n}  f,  \phi),\ee
\[ \tilde c_0 =\frac{ 2^{m(n-m)}\, \pi^{nm/2}\, \Gam_m ((n-m)/2)\,}{\Gam_m (m/2)}.\]
\end{corollary}

\subsection{Analytic properties of the sine transform}

Corollary \ref{85b},  combined with analytic properties of the cosine
transform, enables us to study
analytic continuation of the sine transform
\be\label {lm9} (Q^\a f)(u)=\int_{\vnm}\!\!\!f(v)\,|I_m\!-\!v' uu'
v|^{(\a+m-n)/2}\, d_*v, \quad u \in V_{n,m}.\nonumber\ee

\begin{lemma} \label {mkloi} Let $f$ be an integrable right $O(m)$-invariant function on $\vnm$; $ 2m\le n$, $\om \in C^\infty (\vnm)$. Then  $(Q^\a f, \om)$ extends as a meromorphic function of $\a$  with the polar  set
 $ \{m-1, m-2, \ldots \,\}$
  by the formula
\be \label {gty4e} (Q^{\a}f, \om)=d_\a\, (M^{\a+2m-n} F_{m}f,
\om), \ee
\[ d_\a=\frac{\Gam_{m} (m/2)\, \Gam_{m} (\a/2)}{\Gam_{m}
((n-m)/2)\, \Gam_{m} ((\a+2m-n)/2)}.\]
The normalized function
\be \eta (\a)=\frac{1}{\Gam_{m} (\a/2)}\, (Q^{\a}f, \om)\ee
 extends as an entire function of $\a$.
\end{lemma}
\begin{proof} By  (\ref{gty4}), for $Re \, \a> n-m-1$ we have
$Q^{\a}f=d_\a\, M^{\a+2m-n} F_{m}f$.
 Hence,
\[(Q^{\a}f, \om)=\frac{c\, (M^{\a+2m-n} F_{m}f,
\om)}{\Gam_{m} ((\a+2m-n)/2)}, \qquad c=\frac{\Gam_{m} (m/2)\, \Gam_{m} (\a/2)}{\Gam_{m}
((n-m)/2)}.\]
  Applying Corollary
\ref{cr1} with $\vp=F_{m}f$ and $k=m$, we get the result.
\end{proof}
\begin{remark} {\rm We conjecture that $\eta (\a)$ is an entire function
even if $f$ is not right $O(m)$-invariant, as in the case $m=1$; cf. \cite[p. 474]{Ru4}.}
\end{remark}

\section{Normalized Cosine and Sine Transforms}

For the following it is convenient to suitably normalize  cosine and sine transforms. The normalized transforms will be denoted by the corresponding calligraphic letters.  Assuming $1\le m\le k \le n-1$, we set

\noindent {\bf A.} For $u \!\in\! V_{n,k},\; v \!\in\! V_{n,m}$:

\be \label{n0mby}(\Cs_{m,k}^\a f)(u)=\del_{n,m,k} (\a)\int_{\vnm} \!\!\!f(v)\,
|v'uu'v|^{(\a-k)/2} \, d_*v,\ee

\be \label{n0mbyd}(\ncd0 \vp)(v)=\del_{n,m,k} (\a) \int_{\vnk} \!\!\vp(u)\, |v'uu'v|^{(\a-k)/2} \, d_*u,\ee
\be \label{n0mbs}(\Ss_{m,k}^\a f)(u)=d_{n,m,k} (\a)\int_{\vnm} \!\!\!f(v)\, |I_m
-v'uu'v|^{(\a+k-n)/2} \, d_*v,\ee
\be \label{n0mbsd}(\ncs \vp)(v)=d_{n,m,k} (\a) \int_{\vnk} \!\!\vp(u)\,|I_m
-v'uu'v|^{(\a+k-n)/2}  \, d_*u;\ee
\be \del_{n,m,k}
(\a)=\frac{\Gam_m(m/2)}{\Gam_m(n/2)}\,\frac{\Gam_m((k-\a)/2)}{\Gam_m(\a/2)}, \qquad \a+m-k\neq 1,2, \ldots;
\nonumber \ee
\be d_{n,m,k} (\a)=\frac{\Gam_m(k/2)}{\Gam_m(n/2)}\,\frac{\Gam_m((n-k-\a)/2)}{\Gam_m(\a/2)}, \qquad \a+k+m-n\neq 1,2, \ldots;
\nonumber\ee

\noindent  {\bf B.} For $u \!\in\! V_{n,m},\; v \!\in\! V_{n,m}$:
\be\label{nmn}(\Ms^{\a}
f)(u)=\del_{n,m} (\a)\int_{\vnm} f(v)\, |u'v|^{\a -m} \,
d_*v,\ee
\be\label {44z} (\Qs^\a
f)(u)\!=\!d_{n,m}(\a)
\int_{\vnm}\!\!\!\!\!\!f(v)\,|I_m\!-\!v' uu' v|^{(\a+m-n)/2}\, d_*v,\quad 2m\le n;\ee
\be \del_{n,m}
(\a)=\frac{\Gam_m(m/2)}{\Gam_m(n/2)}\,\frac{\Gam_m((m-\a)/2)}{\Gam_m(\a/2)}, \qquad \a\neq 1,2, \ldots;\nonumber\ee
\[d_{n,m}(\a)=\frac{\Gam_{m} (m/2)\, \Gam_{m} ((n-m-\a)/2)}{\Gam_{m} (n/2)\, \Gam_{m} (\a/2)}, \qquad \a+2m-n\neq 1,2, \ldots;\]
 cf. Remark \ref{lp00}. By Theorem \ref{exi}, all these integrals are absolutely convergent if $Re \, \a> m-1$. Excluded values of $\a$ belong to the polar set  of the corresponding gamma function in the numerator; cf. (\ref{09k}).

\begin{theorem} \label{cr2}

\noindent {\rm (i)} Let $\vp$ be an integrable right
$O(k)$-invariant function on $\vnk$, $\om \in C^\infty
(\vnm)$, $1\le m\le k \le n-1$. Then
\be \label {for1y}  \underset
{\a=0}{a.c.} \,(\ncd0 \vp, \om)\!=\!\mu_k\, (\fd \vp, \om), \quad \mu_k\!=\! \frac{\Gam_{m} (m/2)}{\Gam_{m} ((n\!-\!k)/2)}.\ee
In particular, for  any integrable right
$O(m)$-invariant function $f$ on $\vnm$, $2m\le n$, we have
\be \label {for1t}  \underset
{\a=0}{a.c.} \,(\Ms^{\a}f, \om)\!=\!\mu_m\, (F_m f, \om),\quad \mu_m\!=\! \frac{\Gam_{m} (m/2)}{\Gam_{m} ((n\!-\!m)/2)}.\ee
\end{theorem}
\begin{proof} Equality (\ref{for1y}) follows from (\ref{eq5y}),  (\ref{eq5yq}), and (\ref{n0mbyd}); (\ref{for1t}) is a consequence of (\ref{for1y}).
\end{proof}

The following statement holds by duality.

\begin{corollary} \label{cr124}   Let $f$ be an infinitely differentiable right $O(m)$-invariant function on $\vnm$, $1\le m\le k \le n-1$.
Then for any  $\om \in C^\infty (\vnm)$,
\be \label {for1y4}  \underset
{\a=0}{a.c.} \,(\Cs_{m,k}^\a f, \om)\!=\!\mu_k\, (F_{m,k} f,\om).\ee
\end{corollary}

Theorem \ref {cr2} and Corollary \ref {cr124} show that the Funk transform $F_{m,k}$ and its dual can be regarded (up to a constant multiple) as members of the corresponding analytic families of normalized cosine transforms. If $m=1$, then (\ref{for1y})-(\ref{for1y4})  agree with \cite[Lemma 3.1]{Ru7}.
\begin{lemma}  Let $f$ be an integrable right
$O(m)$-invariant function on $\vnm$,  $1\le m\le k \le n-m$, \[Re\, \a >m-1, \qquad \a\neq k-m+1, k-m+2, \ldots\, .\] Then
\be\label {ores} \ncd0 F_{m,k} f=\varkappa_k \Qs^{\a+n-k-m}f, \qquad \varkappa_k =\frac{ \Gam_m ((n-m)/2)}{\Gam_m (k/2)}.\ee
\end{lemma}
\begin{proof} Write (\ref{gty}) in terms of normalized operators and you are done.
\end{proof}

The next theorem contains inversion formulas for the Funk transform $F_m$ and the cosine transform $\Ms^{\a}$ in terms  of distributions.   The corresponding results for $m=1$ are due to  Semyanistyi  \cite{Se}.
\begin{theorem} \label{cr24} Let $f$ be an integrable right
$O(m)$-invariant function on $\vnm$, $2m \le n$. If  $Re \, \a> m-1$, $ \a\neq m, m+1, m+2, \ldots$, then
  \be \label {for1}  \underset
{\b=2m -\a -n}{a.c.} \,(\Ms^{\b} \Ms^{\a}f, \om)=(f,
\om), \qquad\om \in C^\infty
(\vnm).\ee
 Moreover  \text{\rm (cf. (\ref{for1t}))},
  \be \label {forq1}  \underset
{\a=2m -n}{a.c.} \,(\Ms^{\a} F_m f, \om)\!=\!\varkappa _m (f,
\om),\qquad \varkappa _m \!=\!\frac{ \Gam_m ((n\!-\!m)/2)}{\Gam_m (m/2)}.\ee
\end{theorem}
\begin{proof}  We write (\ref{eq5ya}) in the form
\be\label{eq5yaf}
 \left(\frac{E_{\b-m} M^{\b} g}{\Gam_m (\b/2)}, \F \phi\right)=\check c_2 \, \left (\frac{E_{m-\b-n}  g}{\Gam_m ((m-\b)/2)},  \phi\right),\ee
\[ \check c_2 =\frac{2^{m(n+\b-m)}\, \pi^{nm/2}\,\Gam_m (n/2)}{\Gam_m (m/2)},\]
and set $g= M^{\a}f$. By Theorem \ref {exi}, $g$ is   integrable and right
$O(m)$-invariant  on $\vnm$. Then we compute analytic continuation
 at the point $\b=2m -\a-n$ and obtain
\be\label{1eq5yaf}
 \left(\frac{E_{m-\a-n} M^{2m -\a-n} M^{\a}f}{\Gam_m ((2m \!- \!\a \!- \!n)/2)}, \F \phi\right) \!= \!
 \tilde c_2 \, \left (\frac{E_{\a-m}  M^{\a}f}{\Gam_m ((\a \!+ \!n \!- \!m)/2)},  \phi\right),\ee
\[ \tilde c_2 =\frac{2^{m(m-\a)}\, \pi^{nm/2}\,\Gam_m (n/2)}{\Gam_m (m/2)}.\]
By (\ref{eq5ya}) the right-hand side of (\ref{1eq5yaf}) can be written as
\[
 \frac{\tilde c_2 \, \Gam_m (\a/2)}{\Gam_m ((\a \!+ \!n \!- \!m)/2)}\,\left (\frac{E_{\a-m}  M^{\a}f}{\Gam_m (\a/2)},  \phi\right)=
 c_3\,  \left (\frac{E_{m-\a-n}  f}{\Gam_m ((m-\a)/2)},  \F^{-1}\phi\right),
 \]
where
\[  c_3=   \frac{(2\pi)^{mn}\, \Gam_m^2 (n/2)\, \Gam_m (\a/2)}{\Gam_m^2 (m/2)\, \Gam_m ((\a+n-m)/2)}. \]
To finalize calculations, we set $(\F^{-1}\phi)(x)=(2\pi)^{-mn}(\F\phi)(-x)$ and invoke normalizing factors according to (\ref{nmn}). This yields
\be\label {hhyc}
(E_{m-\a-n} \Ms^{2m -\a-n} \Ms^{\a}f, \F \phi)=(E_{m-\a-n}f, \F \phi),\ee
where both sides  are understood as analytic continuations from the respective domains.
As in the proof
of Corollary \ref{cr1}, we choose $\F\phi$ in
(\ref{hhyc}) so that $(\F\phi)(x)=\psi (r)\om (v)$, $x=vr^{1/2}$, $v\in \vnm$.  We recall that $\om \in C^\infty (\vnm)$ and $\psi (r)$ is a nonnegative $C^\infty$ function on the cone $\Omega$  with compact support away from the boundary $\partial \Omega$.
Then   the left-hand side of (\ref{hhyc}) becomes
\[\underset
{\b=2m -\a-n}{a.c.} \,(E_{\b-m} \Ms^{\b}\Ms^{\a}f, \psi \om)=h_\psi (\a)\underset
{\b=2m -\a-n}{a.c.} \, (\Ms^{\b}\Ms^{\a}f,  \om),\]
\[ h_\psi
 (\a)=2^{-m}\int_{\Omega} |r|^{-(\a+1)/2} \overline{\psi (r)}\,dr,\]
and the right-hand side  equals $h_\psi (\a)\,(f,\om)$.
Owing to uniqueness of analytic continuation, this gives the result.

Let us prove (\ref{forq1}). We set $g=F_mf$ in (\ref{eq5yaf}). Clearly, $g$ is  integrable and right
$O(m)$-invariant; cf. Corollary \ref{laas}. Then we compute analytic continuation
 at the point $\b=2m -n$ and make use of Corollary \ref{nkilz}. This gives
\[ \underset{\b=2m -n}{a.c.} \, \left(\frac{E_{\b-m} M^{\b} F_m f}{\Gam_m (\b/2)}, \F \phi\right)=\check c_0 \, (E_{-m}  F_m f, \phi)=
\check c_{00} \,\, (E_{m-n}  f,  \F \phi),\]
\[ \check c_0 =\frac{2^{m^2}\, \pi^{nm/2}\,\Gam_m (n/2)}{\Gam_m (m/2)\, \Gam_m ((n-m)/2)},\qquad \check c_{00}=\frac{\Gam_m (n/2)}{\Gam_m^2 (m/2)}.\]
 Hence, after normalization,
 \[ \underset{\b=2m -n}{a.c.} \, (E_{\b-m} \Ms^{\b} F_m f, \F \phi)=\varkappa _m  (E_{m-n}  f,  \F \phi), \qquad \varkappa _m =\frac{ \Gam_m ((n-m)/2)}{\Gam_m (m/2)}.\]
 As above, this implies (\ref{forq1}).
 \end{proof}

The following important statement shows that
 the identity operator is a member of the analytic  family $\{\Qs^\a\}$ of the normalized sine transforms,  corresponding to $\a=0$.
\begin{theorem} \label {maeew} Let $f$ be an integrable right $O(m)$-invariant function on $\vnm$; $ 2m\le n$, $\om \in C^\infty (\vnm)$. Then  $(\Qs^\a f, \om)$ extends as a meromorphic function of $\a$  with the polar  set
 $ \{n-2m+1, n-2m+2, \ldots \,\}$.
 Moreover,
\be \label {forq1q}  \underset
{\a=0}{a.c.} \,(\Qs^\a f, \om)=(f,\om).\ee
 \end{theorem}
 \begin{proof} The statement follows from (\ref{gty4e}) and (\ref {forq1})
 if we use normalization  (\ref{44z}) and (\ref{nmn}):
 \bea  \underset
{\a=0}{a.c.} \,(\Qs^\a f, \om)&=&\frac{\Gam_{m}^2 (m/2)}{\Gam_{m} (n/2)}\, \underset
{\b=2m-n}{a.c.} \,\left (\frac{M^{\b} F_{m}f}{\Gam_{m} (\b/2)}, \om \right )\nonumber\\
 &=& \frac{\Gam_{m} (m/2)}{\Gam_{m} ((n-m)/2)}\,\underset
{\b=2m -n}{a.c.} \,(\Ms^{\b} F_m f, \om)=(f,\om).\nonumber\eea
\end{proof}

Equalities (\ref{ores})   and (\ref{forq1q}) imply the following
 inversion result for the Funk transform  $F_{m,k}$.
\begin{theorem} \label{cr24} Let $f$ be an integrable right
$O(m)$-invariant function on $\vnm$, $\om \in C^\infty
(\vnm)$,  $1\le m\le k\le n-m$. Then
\be \label {forq1y}  \underset
{\a=m +k-n}{a.c.} (\ncd0 F_{m,k} f, \om)\!=\!\varkappa_k (f,\om), \qquad \varkappa_k \!=\!\frac{ \Gam_m ((n\!-\!m)/2)}{\Gam_m (k/2)}.\ee
\end{theorem}

 It would be  natural to find an inversion formula
 for the cosine transform $\Cs^\a_{m,k}$, $k>m$, similar to
(\ref {for1}). Regretfully it is not available, however, the following lemma enables us to reduce inversion of  $\Cs^\a_{m,k}$  to the already known procedure for $F_m$; cf. (\ref {forq1}).
\begin{lemma}
\label{cr24p} Let $f$ be an integrable right
$O(m)$-invariant function on $\vnm$, $\;\om \in C^\infty
(\vnm)$,  $1\le m\le k\le n-m$,
\[ Re \,\a >k-1, \qquad \a\neq k, k+1, k+2, \ldots\, . \]
Then
\be \label {forq1y}  \underset
{\b=-\a}{a.c.} (\ncb \Cs^\a_{m,k} f, \om)=\mu_k\,(F_{m}f, \om), \qquad \mu_k\!=\!\frac{\Gam_m (m/2)}{ \Gam_m ((n\!-\!k)/2)}.\ee
\end{lemma}
\begin{proof} By (\ref{gty7}) and (\ref{nmn}),\[
\Ms^{\a+m-k}F_{m}f= c\, \fd \C^\a_{m,k} f, \qquad c=\del_{n,m} (\a+m-k)/\tilde c_\a.\]
Hence, by (\ref{for1}),
\bea
(F_{m}f, \om)&=&\underset
{\b=m-n+k-\a}{a.c.} c\,(\Ms^\b\fd \C^\a_{m,k} f,\om)\nonumber\\&=&\underset
{\b=m-n+k-\a}{a.c.} c\,\del_{n,m} (\b)\,((F_{m,k}M^\b)^*\C^\a_{m,k} f,\om),\nonumber\eea
and (\ref{782}) yields
\[
(F_{m}f, \om)=\underset
{\b=m-n+k-\a}{a.c.} c\,\del_{n,m} (\b)\,c_{n,k,m}(\b)\,(\nck \C^\a_{m,k} f,\om).\]
Now we replace operators on the right-hand side by their normalizations (\ref{n0mbsd}) and (\ref{n0mby}). A simple calculation completes the proof.
\end{proof}

The following ``pointwise'' conjecture looks natural.
\begin{conjecture} \label{8624} Let $f(v)$ be an infinitely differentiable $O(m)$-invariant function on $\vnm$.

\noindent  {\rm (a)} For every $v\in \vnm$, $(\Ms^\a f)(v)$ extends as a meromorphic
function of $\a$ with the polar set $\bbn=\{1,2,3,\ldots\}$.

\noindent  {\rm (b)} For every complex $\a\notin \bbn$,  $(\Ms^\a f)(v)$   is an infinitely differentiable $O(m)$-invariant function on $\vnm$ .

\noindent  {\rm (c)} If $\a \notin \bbn \cup \tilde \bbn, \quad \tilde \bbn = \{2m-n-1,2m-n-2,\ldots\}$, then
 \be (\Ms^{2m-\a-n}\Ms^\a f)(v)=f(v)\ee for every $v\in \vnm$.

 \noindent  {\rm (d)} The Funk transform
  \be \label {la3vf}(F_{m} f) (u)=\int_{\{v\in V_{n, m}: \, u'v=0\}} f(v)\,d_u v,
\qquad u\!\in\! V_{n, m},\ee
  can be inverted by the formula
\be (\Ms^{2m-n}F_m f)(v)=\varkappa _m f(v),\ee
 which holds for every $v\in \vnm$.
\end{conjecture}

In the case $m=1$ these statements are known  and can be obtained using decomposition in spherical harmonics; see, e.g., \cite{Ru2, Ru7}. We expect that Conjecture \ref{8624} can be proved using harmonic analysis developed in \cite{Ge, TT} together with results from \cite{Su}.

\section{The method of Riesz potentials}

We already know that the Funk transform can be formally inverted as
 $F_m^{-1}=\varkappa _m^{-1}\,  \underset
{\a=2m -n}{a.c.} \, \Ms^{\a}F_m f$. There are many ways to realize this analytic continuation. One of them amounts to Blaschke, Radon, Fuglede, and Helgason \cite{Rad, Fug, Hel}  and consists of two steps: (a) We apply a certain back-projection operator and  reduce the problem to inversion of the corresponding potential; (b) We invert that potential operator.
In the case of the Funk transform on the sphere, the potential operator is realized as  the sine transform, explicit inversion of which is pretty sophisticated; see, e.g., \cite{Hel, Ru4}. In the higher-rank case the situation  is much more complicated; cf. \cite[Theorem 3.4]{Gri}, \cite[Theorem 10.4]{Ka}.
  Semyanistyi \cite {Se}  suggested to express the spherical sine transform  through the spatial Riesz potential. In this section we generalize his idea for the higher-rank case.

  There is a substantial difference between  the cases, when $n-k-m$ is an even number and  an odd number.  This phenomenon, which is well-known in the  case $m=1$, is related to the so-called local and non-local inversion formulas.

\subsection {Local inversion formulas for the Funk transform}

It is instructive to start with the particular case $k=m$, when evenness of $n-k-m$ is equivalent to the evenness of $n$.
\begin{theorem}\label{reven}
Let  $f$ be a right $O(m)$-invariant function in $L^1 (\vnm)$, $ 2m
\le n$. If $n$ is even and
\[c_m=\frac{(-1)^{m(n/2 -m)}\,2^{m(2m-n)} \,\Gam^2_m(m/2)}{\Gam^2_m((n-m)/2)},\]
then the Funk
 transform $\vp=F_m f$ can be inverted by the formula
\be \label {pra} E_{m-n} f=c_m\,\Del^{n/2 -m} \,E_{-m}F_m \vp\qquad \mbox{\rm (in
the $S'$-sense).}\ee  If the right-hand side of (\ref{pra}) is locally
integrable in the neighborhood of the Stiefel manifold $V_{n,m}$,
then  (\ref{pra}) holds pointwise a.e. on $\vnm$.
\end{theorem}
\begin{proof} By (\ref{ceq5yq}) for $\phi \in S(\Ma)$ we have
\[(-1)^{ m \ell}(\Del^\ell \,E_{-m}F_m \vp, \hat \phi)\!=\!\tilde c_0 \,(E_{m-n} \vp,|x|_m^{2\ell}\phi)\!=\!\tilde c_0 \,(E_{m-n+2\ell} F_m f,\phi). \]
We can choose $2\ell =n-2m$ to get $\tilde c_0 \,(E_{-m}
F_mf,\phi)$. By  (\ref{ceq5yq})
 this coincides with $\tilde c_0 ^2\,(2\pi)^{-nm}\,(E_{m-n} f,\hat\phi)$, and the result follows.
\end{proof}

The next statement is more general.
\begin{theorem}\label{kreven}
Let  $f$ be a right $O(m)$-invariant function in $L^1 (\vnm)$, $ 1\le m\le k \le n-m$.
 If $n-k-m$ is an even number and
\[c_{m,k}=\frac{(-1)^{m\ell}\,4^{-m\ell} \,\Gam_m(m/2)\,\Gam_m(k/2)}{\Gam_m((n-m)/2)\, \Gam_m((n-k)/2)},\quad \ell=(n-k-m)/2,\]
then the Funk
 transform $\vp=F_{m,k}f$ can be inverted by the formula
\be \label {prak} E_{m-n} f=c_{m,k}\,\Del^{\ell} \,E_{-k}\fd \vp\qquad \mbox{\rm (in
the $S'$-sense).}\ee If the right-hand side of (\ref{prak}) is locally
integrable in the neighborhood of  $V_{n,m}$,
then  (\ref{prak}) holds pointwise a.e. on $\vnm$.
\end{theorem}
\begin{proof} By (\ref{ceq5yf}) and (\ref{gty7}),
$$
(-1)^{ m \ell}(\Del^\ell \,E_{-k}\fd \vp, \hat \phi)=d _0 \,  \underset
{\a=0}{a.c.} \,\left (\frac{E_{\a -k}  \cd0 \vp}{\Gam_{m} (\a/2)},  (-1)^{ m \ell} \Del^\ell \hat\phi\right)$$
$$= d_1\,  \underset
{\a=0}{a.c.} \,\left (\frac{E_{\a -k}  M^{\a+m-k}
F_{m}f}{\Gam_{m} ((\a+m-k)/2)},  (-1)^{ m \ell} \Del^\ell \hat\phi\right),$$
$$ d_1= \frac{\Gam_m ((n-k)/2)\,\Gam_{m} (m/2)\,}{\Gam_m (n/2)}.$$
Owing to (\ref{eq5ya}), for $2\ell =n-k-m$, this expression can be written as
\[d_2\,  \underset
{\a=0}{a.c.} \,(E_{k-n-\a} F_{m}f, |x|_m^{2\ell}\phi)=d_2\,(E_{-m}
F_mf,\phi), \]\[ d_2=2^{m(n-k)}\, \pi^{nm/2}\,\Gam_m ((n-k)/2)/\Gam_{m} (k/2).\]
Finally, by (\ref{ceq5yq}),
\[
(-1)^{ m \ell}(\Del^\ell \,E_{-k}\fd \vp, \hat \phi)=\tilde c_0\, d_2\,(E_{m-n} f,\check\phi)= \frac{\tilde c_0\, d_2}{(2\pi)^{nm}}\,(E_{m-n} f,\hat\phi).\]
This gives the result.
\end{proof}

\subsection {Sine transforms and non-local inversion formulas}

Theorems \ref{reven} and  \ref{kreven} show that when $n-k-m$ is  odd, we have to deal with  fractional powers of the
operator $\Del$. The latter are realized by the
Riesz potential  $I^\a$; see Section \ref{se4}.  For the following we need some more preparation.

\subsubsection {The Semyanistyi-Lizorkin-Samko spaces $\Phi_V$}

 From the Fourier
transform formula (\ref{fou})
 it is evident that the Schwartz class $\S$ is not well-adapted for Riesz potentials,
 because $\S$  is not invariant under multiplication by
$|y|_m^{-\a}$. To get around this difficulty, we choose another
space of test functions.
  Let $\Psi=\Psi(\Ma)$ be the collection
of all functions $\psi (y) \in \S(\Ma)$ vanishing on the manifold
\be\label{sets} V \!= \!\{y\in \Ma : \rank (y)<m \} \! =
\!\{y\in \Ma : |y'y| \!= \!0 \}\ee with all derivatives. The manifold $V$ is a cone in $\bbr^{nm}$ with
vertex $0$.  Let $\Phi =\Phi(\Ma)$ be the
 Fourier image of $\Psi$. Since the Fourier transform
 is an automorphism of $\S$, then $\Phi$ is a closed linear subspace of $\S$, which is isomorphic to $\Psi$.

The spaces  $\Phi$ and $\Psi$ were introduced by V.I. Semyanistyi
[Se] in the case $m=1$.  They have proved to be very useful
in integral geometry and real analysis. Further generalizations and applications are mainly due to Lizorkin and Samko; see  \cite{Li, Ru96, Sa1, SKM} on this subject.

The following characterization of the space $\Phi$ is a consequence
of a  more general result by Samko \cite{Sa1}.
\begin{theorem} The Schwartz function $\phi (x)$ on $\Ma$ belongs
to the space $\Phi$ if and only if it is orthogonal to all
polynomials  $p(x)$ on any hyperplane $\t$ in $\bbr^{nm}$ having the
form $\t=\{x: \tr (a'x)=c \}, \; a \in V$: \be \int_{\t} p(x) \phi
(x)d\mu(x)=0,\ee $d\mu(x)$ being the induced Lebesgue measure on
$\t$.
\end{theorem}

When $m=1$, the space $\Phi$ consists of Schwartz functions
which are orthogonal to all polynomials on $\bbr^{n}$.

 We denote by  $\Phi'$ the space of all semilinear
continuous functionals on $\Phi$. Two $\S'$-distributions that coincide in the
$\Phi'$-sense, differ from each other by an arbitrary
$\S'$-distribution with the Fourier transform supported by $V$.
Since for any complex $\a$, multiplication by $|y|_m^{-\a}$ is an
automorphism of $\Psi$, then, according to the general theory
\cite{GSh2}, $I^\a$ is an automorphism of $\Phi$, and we have \be
\label{fof}\F[I^\a \phi](y)=|y|_m^{-\a}\F[\phi](y), \qquad \phi \in
\Phi.\ee This
 gives  \be
\label{ff}\F[I^\a f](y)=|y|_m^{-\a}\F[f](y)\ee for any
$\Phi'$-distribution $f$.
For $k$ even, the Riesz potential $I^kf$ can be
inverted  by repeated
application  of the Cayley-Laplace operator $ \Del$ in the sense of $\Phi'$-distributions.

\subsubsection{Inversion of the sine transform}

\begin{lemma}\label{re7k} Let $f $ be  an integrable right $O(m)$-invariant function on
 $\vnm$, $2m\le n$. Then for  any $\a \in \bbc$,
\be\label {nn8} a.c. \, E_{\a+m-n}\Qs^\a f= 2^{\a m}\,  a.c. \, I^\a
E_{m-n} f \ee in the $\Phi'$-sense.
\end{lemma}
\begin{proof} Let $\phi \in
\Phi,\; \psi=\check \phi \in \Psi$. We first suppose that \be\label
{038} Re \, \a>n-m-1; \qquad \a\!\neq\! n\!-\!m,  n\!-\!m\!+\!1,
n\!-\!m \!+\!2, \ldots .\ee Then, owing to (\ref {44z}),
(\ref{gty4}), and (\ref{eq5ya}),  \bea \label {oiy}\qquad
(E_{\a+m-n}\Qs^\a f,\phi)&=&d_{n,m}(\a)\, d_\a (E_{\a+m-n}M^{\a+2m
-n}
F_{m}f, \hat \psi)\nonumber\\
&=& \label {krob}\frac{2^{(\a+m)m}\, \pi^{nm/2}\, \Gam_m (m/2)}{
\Gam_m ((n-m)/2)}\,(E_{-m}F_{m}f,|x|_m^{-\a}\psi).\eea Since
$|x|_m^{-\a}\psi (x)\!=\!(2\pi)^{-nm}|x|_m^{-\a}\hat \phi
(-x)\!=\!(2\pi)^{-nm} (I^\a \phi)^\wedge (-x)$, then (\ref{ceq5yq})  yields
\[ (E_{\a+m-n}\Qs^\a f,\phi)=2^{\a m}\,(E_{m-n}f,I^\a
\phi)= 2^{\a m}\,(I^\a E_{m-n}f,\phi).\] An expression in
(\ref{krob}) is an entire function of $\a$. Hence, the result
follows by analytic continuation.
\end{proof}

Lemma \ref{re7k} enables us to reconstruct  $f$ from $Q^\a f$ in the $\Phi'$-sense.
\begin{corollary}\label {kkli}  Let $g=Q^\a f$, where $f $ is  an integrable right $O(m)$-invariant function on
 $\vnm$, $2m\le n$,
 \[Re \, \a >m-1; \qquad \a\!\neq\! n\!-\!2m\!+\!1, n\!-\!2m \!+\!2,
\ldots . \] We set $\a=2\ell -\gam$, where $\ell$ is a positive
integer. The following inversion formula holds in the
$\Phi'$-sense: \be E_{m-n} f=2^{-\a m}\,  (-1)^{\ell m}\, \Del ^\ell
I^\gam E_{\a+m-n}Q^\a f,\ee where $\Del$ is the Cayley-Laplace
operator (\ref{K-L}).
\end{corollary}
\begin{remark} {\rm
An analogue of (\ref{nn8}) for $m=1$ is contained in Lemma 3.1 from
\cite{Ru4} and invokes spherical convolutions with hypergeometric
kernel. It is an interesting open problem to extend that Lemma  to
the higher-rank case.}
\end{remark}

\subsubsection {Non-local inversion of cosine and Funk transforms}

\begin {theorem}  \label{cr2t} Let $f$ be an integrable right
$O(m)$-invariant function on $\vnm$, $1\le m\le k\le n-m$, and let $\ell$ be a positive
integer.

\noindent {\rm (i)} If $g=\Cs^\a_{m,k} f$,
  $$Re \, \a> m-1, \qquad
\a+m-k \neq 1, 2, \ldots \,,$$
 and  $\a+n-k-m=2\ell -\gam$, then the following inversion formula holds in the
$\Phi'$-sense:
\be  \label {441} E_{m-n} f=c\, \Del^\ell
I^\gam E_{\a-k}\fd \Cs^\a_{m,k} f,\ee
\[ c=\frac{2^{m(k+m-n-\a)}(-1)^{m \ell}\, \Gam_m (k/2)}{\Gam_m ((n-m)/2)}.\]

\noindent {\rm (ii)} If $\vp=F_{m,k} f$ and  $n-k-m=2\ell -1$, then,  similarly,
\be  \label {441f} E_{m-n} f=c_*\, \Del^\ell
I^1 E_{-k}\fd F_{m,k} f,\ee
\[ c_*=\frac{2^{m(k+m-n)}(-1)^{m \ell}\, \Gam_m (k/2)\,  \Gam_m (m/2)}{\Gam_m ((n-m)/2)\, \Gam_m ((n-k)/2)}.\]
\end{theorem}
\begin {proof}
(i) By (\ref {gty}), after normalization, we obtain
 \be \label{lpby}\fd \Cs_{m,k}^\a f=\tilde c_\a\, \Qs^{\a+n-k-m} f,\qquad \tilde c_\a=\frac{c_\a \, \del_{n,m,k}(\a)}{d_{n,m} (\a+n-k-m)}.\nonumber\ee
Then we apply Corollary \ref{kkli} with $\a$ replaced by $\a+n-k-m$. This gives the result.

(ii) By (\ref{eq5y}) and (\ref{eq5yq}), for any $h \in \S(\Ma)$ we have
\be\label {xder}\underset
{\a=0}{a.c.} \,
 \left(\frac{E_{\a-k} \cd0 \vp}{\Gam_m (\a/2)}, h \right)=\lam_1 \, (E_{-k}\fd \vp, h ),\ee\[ \lam_1 = \frac{\Gam_m (n/2)}{\Gam_m (k/2)\,\Gam_m ((n-k)/2)}.\]
Owing to (\ref {gty}), the left-hand side of (\ref{xder}) can be written as
\[\underset
{\a=0}{a.c.} \, \frac{c_\a}{\Gam_m (\a/2)}\, (E_{\a-k} Q^{\a+n-k-m} f, h ).\]
Now we replace  $h$ by the Riesz potential $I^{1-2\ell}\phi=(-1)^{m \ell}\Del^\ell I^1\phi$, where $\phi\in \Phi$, $2\ell-1=n-k-m$, and apply (\ref{nn8}). This gives
\[\lam_1  \, (-1)^{m \ell}\,(\Del^\ell I^1 E_{-k}\fd \vp, \phi)\!=\!\lam_2\, \underset
{\a=0}{a.c.} \,(I^\a E_{m-n} f, \phi)\!=\!\lam_2\, (E_{m-n}f, \phi),\]\[ \lam_2 = \frac{2^{m(n-m-k)}\Gam_m (n/2)\, \Gam_m ((n-m)/2)}{\Gam_m^2 (k/2)\,\Gam_m (m/2)},\]
and the result follows.
 \end{proof}

\section{Appendix}

\subsection{On  the paper by Gelfand,   Graev, and Rosu}\label {bhub}
Sections 2 and 3 of \cite{GGR} deal with analytic continuation of a certain auxiliary operator $R_p^\lam$; see, e.g.,  formula (3.5) in \cite{GGR}. Ignoring the normalizing factor, and changing notation $p=m, \;\lam=\a+n-m$, one can write it as a cosine transform
\be \label
{nkmap}(M^{\a} f)(u)=\int_{\vnm} \!\!\!f(v)\,
|u'v|^{\a-m} \, d_*v,\qquad u\in\vnm,\ee
where $2m\le n$, $Re \,\a >m-1$, and $f$ is a right $O(m)$-invariant function. It was stated on p. 368, that elementary computation (which is skipped) leads to the formula of $R_p^\lam$ in coordinates; see  formula (3.6) in the same paper. Below we perform computation  and arrive at a different expression,  which does not fall (at least,  directly) into the scope of references \cite{P1, Rai}, as stated in \cite[Proposition 1, p. 368]{GGR}.\footnote{The reasoning from \cite{P1} is reproduced in Lemma \ref{l3.1} for  the more general situation.}

Let ${\tilde\frM}_{n-m,m}$ be the subset of $\frM_{n-m,m}$, which  consists of matrices of rank $m$. Every matrix $y\in{\tilde\frM}_{n-m,m}$ is uniquely represented in polar coordinates as $y=\om s^{1/2}$, $\om \in V_{n-m,m}$, $s\in \Omega$. We set
\be \label {8gbu} v=\left[\begin{array} {c} \om((I_m+s)^{-1}s)^{1/2} \\ (I_m+s)^{-1/2} \end{array}
 \right]=\left[\begin{array} {c} v_1 \\ v_2 \end{array}
 \right].\ee
 Since \bea v'v&=& ((I_m+s)^{-1}s)^{1/2}\om'\om ((I_m+s)^{-1}s)^{1/2}+(I_m+s)^{-1}\nonumber\\&=&(I_m+s)^{-1}s+(I_m+s)^{-1}=I_m, \nonumber\eea then
 $v\in \vnm$. Thus (\ref{8gbu}) defines a map $\mu:{\tilde\frM}_{n-m,m} \to \vnm$.  Conversely, given $v\in \vnm$, the corresponding matrix $y=\om s^{1/2}\in {\tilde\frM}_{n-m,m}$ can be reconstructed from obvious relations
 \be\label{mklr}
(I_m+s)^{-1}=v'v_0v'_0v, \qquad v_0=\left[\begin{array} {c}0 \\ I_m \end{array}
 \right]\in \vnm,\ee
\[ \om((I_m+s)^{-1}s)^{1/2}=\check v'_0v, \qquad \check v_0=\left[\begin{array} {c} I_{n-m} \\ 0 \end{array}
 \right]\in V_{n,n-m}.\]

\begin{lemma} \label{mkze} Let $2m\le n$. If $f$ is an integrable  right $O(m)$-invariant function on $\vnm$,  $\tilde f(y)=(f\circ\mu)(y)$, then
\be\label {9nyf}\int_{\vnm} f(v)\, dv= \sig_{m,m}\,\int_{{\frM}_{n-m,m}}\frac{\tilde f(y)}{|I_m +y'y|^{n/2}}\, dy.\ee
If $F$ is  an integrable  function on ${\frM}_{n-m,m}$, $\check F(v)=(F\circ\mu^{-1})(v)$, then
\be\label {9nyfb}\int_{{\frM}_{n-m,m}}F(y)\, dy=\frac{1}{\sig_{m,m}}\,\int_{\vnm} \frac{\check F(v)}{|v'_0v|^n}\, dv. \ee
\end{lemma}
\begin{proof} We denote by $I$ the left-hand side of (\ref{9nyf}). By (\ref{2.11}) with $k=n-m$,
\[ I= \intl_0^{I_m} d\nu(r) \intl_{ V_{n-m, m}}dw
 \intl_{ V_{m, m}} f\left(\left[\begin{array} {cc} wr^{1/2} \\ u(I_m -r)^{1/2} \end{array}
 \right]\right) \ du,
\]
 \[d\nu(r)=2^{-m} |r|^{(n-2m-1)/2}|I_m -r|^{-1/2}\, dr.\]
Changing variable $r=I_m- (I_m+s)^{-1}$, we get
\[I=2^{-m}\int_{\p}d\tilde \nu(s)\intl_{ V_{n-m, m}}dw
 \intl_{ V_{m, m}} f\left(\left[\begin{array} {c} \om((I_m+s)^{-1}s)^{1/2} \\ u(I_m+s)^{-1/2} \end{array}
 \right]\right) \ du,
\]
 \[d\tilde \nu(s)=2^{-m} |s|^{(n-2m-1)/2}|I_m+s|^{-n/2}\,ds.\]
 Since $f$ is right $O(m)$-invariant,  then integration over $V_{m, m}=O(m)$ can be suppressed and Lemma \ref {l2.3} yields (\ref{9nyf}). The second equality is a consequence of the first one, owing to (\ref{mklr}).
\end{proof}

Now we return back to (\ref{nkmap}). In the statement below $v_0, f,  \tilde f, F$, and $\check F$ have the same meaning as in Lemma \ref{mkze};
 $u,v\in\vnm$; $x,y\in {\frM}_{n-m,m}$; $u=\mu (x)$. We set
\[ \tilde K_\a (x,y)=\frac{|I_m+x'x|^{(m-\a)/2}}{|I_m+y'y|^{(n+\a-m)/2}}, \qquad \check K_\a (u,v)=\frac{|v'_0u|^{m-\a}}{|v'_0v|^{n+\a-m}}.\]
\begin{lemma} \label {oontr} The following relations hold provided that integrals in either side are absolutely convergent:
\be\label {8mi8}\intl_{\vnm} \!\!\!f(v)\,
|u'v|^{\a-m} \, d_*v\!=\!\frac{\sig_{m,m}}{\sig_{n,m}}\intl_{{\frM}_{n-m,m}} \!\!\!\!\!\tilde f(y)\,|I_m\!+\!x'y|^{\a-m} \tilde K_\a (x,y)\, dy,\ee
\be\label {8mi82}\intl_{{\frM}_{n-m,m}}\!\!\! \!\!F(y)\,|I_m\!+\!x'y|^{\a-m}\, dy\!=\!\frac{\sig_{n,m}}{\sig_{m,m}}\intl_{\vnm} \!\!\!f(v)\,
|u'v|^{\a-m} \check K_\a (u,v)\, d_*v.\ee
\end{lemma}
\begin{proof} As above, we set $x=\theta r^{1/2}$, $y=\om s^{1/2}$, where $\theta,\om \in V_{n-m,m}$ and  $r,s\in \Omega$. Then we define $v\in \vnm$ by (\ref{8gbu}) and, similarly,
\be \label {8gbuu} u=\left[\begin{array} {c} \theta((I_m+r)^{-1}r)^{1/2} \\ (I_m+r)^{-1/2} \end{array}
 \right]=\left[\begin{array} {c} u_1 \\ u_2 \end{array}
 \right].\ee
Keeping in mind that
\[r(I_m +r)^{-1}=(I_m +r)^{-1}r, \qquad ((I_m +r)^{-1}r)^{1/2}=(I_m+r)^{-1/2}r^{1/2},\]
we easily have
\[|u'v|=|u_1'v_1 +u'_2v_2|=|I_m+r|^{-1/2} |I_m+r^{1/2}\theta'\om s^{1/2}| |I_m+s|^{-1/2}.\]
Hence, by (\ref{mklr}),
\[ |u'v|=\frac{ |I_m+x'y| }{|I_m+x'x|^{1/2}\, |I_m+y'y|^{1/2}},\qquad |I_m+x'y| =\frac{|u'v|}{|v'_0u|\,|v'_0v|}.\]
It remains to plug these expressions in the corresponding integrals and apply Lemma \ref{mkze}.
\end{proof}
\begin{remark} {\rm An important factor $\tilde K_\a (x,y)$ in (\ref{8mi8}), that depends on $\a$ and  affects the  behavior   at infinity, is skipped in \cite[formula (3.6)]{GGR}. Moreover, Lemma \ref{mkze} reveals that, in general,   $\tilde f(y)$ is not rapidly decreasing. It follows that the reasoning from \cite {P1, Rai}, related to
 analytic continuation of the distribution $|x|^\lam/\Gam_p ((\lam+p)/2)$ cannot be directly applied to  (\ref{8mi8}). To understand the essence of the matter, the reader is encouraged to   perform  analytic continuation in detail for $m=1$. Note also that, the integral operator on the right-hand side of (\ref{8mi82}) is not $O(n)$-invariant,  unlike the left-hand side of (\ref{8mi8}).

 In the preceding work \cite{GGS}, the result was deduced from \cite{GGS67}, using transition to the Radon transform on the space of matrices. This transition requires careful inspection because it does not lead to rapidly decreasing functions, which are assumed in \cite{GGS67}.}
\end{remark}

\subsection{A useful integral}

\begin{lemma}\label {exl} Let $\lv=(\lam_1, \dots ,\lam_m) \in
\bbc^m$, $\; u\in \vnk, \quad v \in \vnm$,  $1\le m\le k\le n$. Then
\be \label {ave} \intl_{\vnm} (v'uu'v)^{\lv}\, d_*v= \intl_{\vnk} (v'uu'v)^{\lv}\, d_*u
= \frac{\Gam_m (m/2)\,\Gam_\Om(\lv+\kv)}{\Gam_m (k/2)\, \Gam_\Om(\lv+\nv)}. \ee This integral converges
absolutely if and only if $Re \, \lam_j > j-k-1$ for each $j=1,2,
\ldots ,m$.
\end{lemma}
\begin{proof} The first equality follows from Lemma \ref{hbi1}.
Both integrals  are, in fact, constants.  Thus, it suffices to evaluate
\be \label {ax} I=\intl_{\vnm} (v'u_0 u_0'v)^{\lv} \,
d_*v, \qquad u_0=
\left[\begin{array} {c}  I_{k} \\
0 \end{array} \right] \in \vnk. \ee Consider
 an auxiliary integral
\be \label {aax} A=\intl_{\Ma} (x'u_0 u_0'x)^{\lv} \, e^{-{\rm tr}
(x'x)} \, dx.\ee
 We compute it in two different ways. Let first $$x=\left[\begin{array} {c}
  a \\ b \end{array} \right], \qquad a \in \frM_{k,m}, \quad b\in
  \frM_{n-k,m}.$$ Then $u_0'x=a, \, x'x= a'a +b'b$, and we have
\be \label {q1} A\!=\!A_1 A_2, \quad A_1\!=\!\!\intl_{\frM_{k,m}}\!\! (a'a)^{\lv}
\, e^{-{\rm tr} (a'a)} \, da, \quad A_2\!=\!\! \intl_{\frM_{n-k,m}}\!\!
e^{-{\rm tr} (b'b)} \, db.\ee
Passing to polar coordinates, owing to (\ref{gf}) and (\ref{det}), we obtain \be \label {q2} A_1=2^{-m}
\sig_{k,m} \intl_{\Omega} r^{\lv+\kv} e^{-{\rm tr} (r)} d_{*} r=
2^{-m}
\sig_{k,m}\, \Gam_\Om(\lv +\kv), \ee provided
$Re \, \lam_j
> j-k-1$,  $\, \forall j=1,2, \ldots ,m$. The last condition is sharp. It gives
the ``only if" part of the lemma. For
$A_2$ we have
\[ A_2=\left
( \, \intl_{-\infty}^{\infty}e^{-s^2}\ ds \right )^{m(n-k)}=
\pi^{m(n-k)/2}.\] Thus \be \label {q22} A=\frac{ \pi^{nm/2}\,
\Gam_\Om(\lv +\kv)}{\Gam_m (k/2)}, \qquad Re \, \lam_j
> j-k-1.\ee

On the other hand, by setting $x=vt, \, v \in \vnm,   \, t \in T_m$,
owing to Lemma \ref{sph}, we obtain
\[A=\intl_{T_m} e^{-{\rm tr} (t't)} \, d \mu (t)\intl_{\vnm}
(t'v'u_0 u_0'vt)^{\lv} \, dv,\]
\[ d \mu (t)=\prod\limits_{j=1}^m t_{j,j}^{n-j} \, dt_{j,j} \, dt_*, \qquad
dt_*=\prod\limits_{i<j} dt_{i,j}.\] By (\ref{pr6}), one can write
\be \label {q3}A=BI,\ee where $I$ is our integral (\ref{ax})
 and
\bea B&=&\sig_{n,m}\intl_{T_m} (t't)^\lv e^{-{\rm tr} (t't)}
\,d\mu(t)=\nonumber
\\ &=&\sig_{n,m}\,\prod\limits_{j=1}^m  \intl_0^\infty t_{j,j}^{\lam_j+n-j}\,
e^{-t_{j,j}^2}\, dt_{j,j} \times \prod\limits_{i<j}\,
\intl_{-\infty}^\infty e^{-t_{i,j}^2}\, dt_{i,j} \nonumber \\
&=&2^{-m} \,\pi ^{m(m-1)/4} \,\sig_{n,m} \prod\limits_{j=1}^m \Gam \Big
(\frac{\lam_j+n-j+1}{2}\Big )\nonumber \\ &=&2^{-m}
\,\sig_{n,m}\,\Gam_\Om(\lv+\nv), \qquad Re \, \lam_j > j-n-1.\nonumber \eea
Combining this with (\ref{q22}) and (\ref {q3}), we obtain
\[ I=\frac{A_1 A_2}{B}= \frac{\Gam_m (m/2)\,\Gam_\Om(\lv+\kv)}{\Gam_m (k/2)\, \Gam_\Om(\lv+\nv)}.\]
\end{proof}

In the case $\lam_1=\ldots=\lam_m=\lam$ we have the following.
\begin{corollary}\label{le} Let $1\le m\le k\le n$,
$\;Re\, \lam > m-k-1$. Then
\be\label{mnv}
 \intl_{V_{n, m}}|v'uu'v|^{\lam/2} \,
dv\!=\!\intl_{V_{n, k}}|v'uu'v|^{\lam/2} \,
du\!=\!\frac{\Gam_m (n/2)\, \Gam_m ((\lam\!+\!k)/{2})}{\Gam_m (k/2)\,\Gam_m ((\lam\!+\!n)/{2})}. \ee  \end{corollary}

\bibliographystyle{amsalpha}

\end{document}